\pdfoutput=1
\documentclass[10pt,oneside]{amsart}
\usepackage[
paper=a4paper,
headsep=20pt,text={136mm,208mm},centering,includehead
]{geometry}
\usepackage{graphicx}
\usepackage[dvipsnames,usenames]{xcolor}
\usepackage[colorlinks=true,urlcolor=MidnightBlue,linkcolor=MidnightBlue,citecolor=MidnightBlue]{hyperref}
\hypersetup{
	linkbordercolor={1 0 0}, 
	citebordercolor={0 1 0} 
}
\usepackage{amssymb,mathtools,mathrsfs,stmaryrd,wasysym,bbm,verbatim,faktor,cite}
\usepackage[UKenglish]{babel}
\usepackage[noabbrev,sort]{cleveref}
\usepackage{tikz}
\usetikzlibrary{trees,decorations.pathmorphing,decorations.markings,matrix,shapes}
\usepackage[backgroundcolor=green,linecolor=green,bordercolor=green]{todonotes}
\usepackage[shortlabels]{enumitem}
\setlist[enumerate]{itemsep=5pt}
\usepackage[outdir=./]{epstopdf}

\usepackage{float}

\newtheorem{theorem}{Theorem}[section]
\newtheorem{definition}[theorem]{Definition}
\newtheorem*{theorem*}{Theorem}
\newtheorem*{corollary*}{}
\newtheorem{proposition}[theorem]{Proposition}
\newtheorem{corollary}[theorem]{Corollary}
\newtheorem{lemma}[theorem]{Lemma}

\newtheorem{question}{Question}
\newtheorem{example}[theorem]{Example}
\newtheorem{remark}[theorem]{Remark}

\numberwithin{equation}{section}

\makeatletter

\providecommand{\proofname}{Proof}
\makeatother

\newcommand{\dkh}{DKh}

\newcommand{\K}{\mathbb{K}}
\newcommand{\Kinv}{\mathbb{K}^{-1}}
\newcommand{\G}{\mathcal{G}}
\newcommand{\BG}{\mathbb{G}}

\newcommand{\Z}{\mathbb{Z}}
\newcommand{\graphene}{\mathcal{G}}
\newcommand{\BL}{\mathbb{L}}
\newcommand{\ra}{\rightarrow}

\newcommand{\CGlyph}{\raisebox{-0.25\height}{\includegraphics[width=0.5cm]{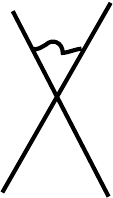}}}
\newcommand{\VGlyph}{\raisebox{-0.25\height}{\includegraphics[width=0.5cm]{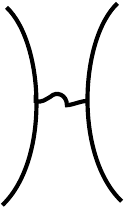}}}
\newcommand{\YGlyph}{\raisebox{-0.25\height}{\includegraphics[width=0.5cm]{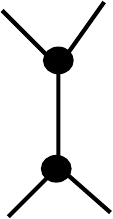}}}

\begin{document}

\title{On ribbon graphs and virtual links}

\author{Scott Baldridge}
\address{Department of Mathematics, Louisiana State University\\
Baton Rouge, LA}
\email{sbaldrid@math.lsu.edu}

\author{Louis H. Kauffman}
\address{Department of Mathematics, Statistics and Computer Science\\ 851 South Morgan Street\\ University of Illinois at Chicago\\
Chicago, Illinois 60607-7045} 

\address{Department of Mechanics and Mathematics\\
Novosibirsk State University\\
Novosibirsk, Russia}
\email{kauffman@uic.edu}	
	
\author{William Rushworth}
\address{Department of Mathematics, Syracuse University}
\email{wrushwor@syr.edu}	

\def\subjclassname{\textup{2020} Mathematics Subject Classification}
\expandafter\let\csname subjclassname@1991\endcsname=\subjclassname
\expandafter\let\csname subjclassname@2000\endcsname=\subjclassname
	

\begin{abstract}
We introduce a new equivalence relation on decorated ribbon graphs, and show that its equivalence classes directly correspond to virtual links. We demonstrate how this correspondence can be used to convert any invariant of virtual links into an invariant of ribbon graphs, and vice versa.
\end{abstract}
	
\maketitle

\section{Introduction}\label{Sec:intro}
\noindent In the traditional relationship between graph theory and knot theory graphs are intermediate objects in the study of links: a graph is constructed from a link diagram, which is then used to define an invariant of the link represented by the diagram. In this paper we introduce a construction that allows one, among other things, to invert the traditional relationship, using links to define new invariants of graphs. This construction also reinforces the traditional relationship, allowing for greater fidelity when transferring graph-theoretic constructions to knot theory.

\subsection{Correspondence between ribbon graphs and virtual links}\label{Sec:correspondence}
This paper is centred around the introduction of a direct correspondence between equivalence classes of decorated trivalent ribbon graphs and virtual links. By construction, given an invariant of virtual links we may use this correspondence to pull it back to an invariant of decorated graphs. Such invariants often contain information that is independent of the decoration placed on the graph.

This correspondence has its genesis in work of the first author in \cite{Baldridge2018}. Baldridge defined a homology theory of trivalent ribbon graphs with a perfect matching, and conjectured that it could be used to count the number of Tait colorings of the argument graph. This was verified in \cite{BaldridgeLowranceMcCarty2018}, establishing that the homology theory contains information regarding the abstract graph (independent of the choice of perfect matching).

It was recognized by the authors of the present work that the homology theory defined by Baldridge bears a striking similarity to the Khovanov homology of virtual links \cite{Baldridge2018}\footnote{Adam Lowrance also noticed the similarity.}. Virtual links are a generalization of classical links in the \(3\)-sphere, and Khovanov homology is a powerful group-valued invariant of such links. As we demonstrate in \Cref{Sec:graphenes}, one may convert a decorated ribbon graph into a virtual link diagram using the following replacement at matched edges
\begin{equation}\label{Eq:kmap}
	\raisebox{-20pt}{\includegraphics[scale=0.75]{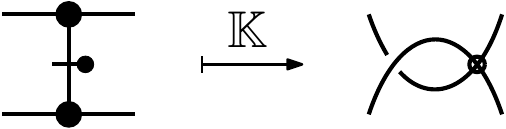}}
\end{equation}
In this paper we demonstrate that the replacement depicted above defines a functor, \(\K \), that allows any invariant of virtual links to be converted into an invariant of decorated ribbon graphs, and vice versa.

Links in $S^3$ and virtual links are described combinatorially via diagrams and the Reidemeister moves (see \Cref{Fig:crms}). We give a new equivalence relation on graphs that respects the Reidemeister moves under \( \K \). It follows that \( \K \) uniquely associates a virtual link to an equivalence class of decorated ribbon graphs. For example, the trefoil is associated to the \( K_{3,3} \) graph (with a certain perfect matching):
\begin{center}
\includegraphics[scale=0.5]{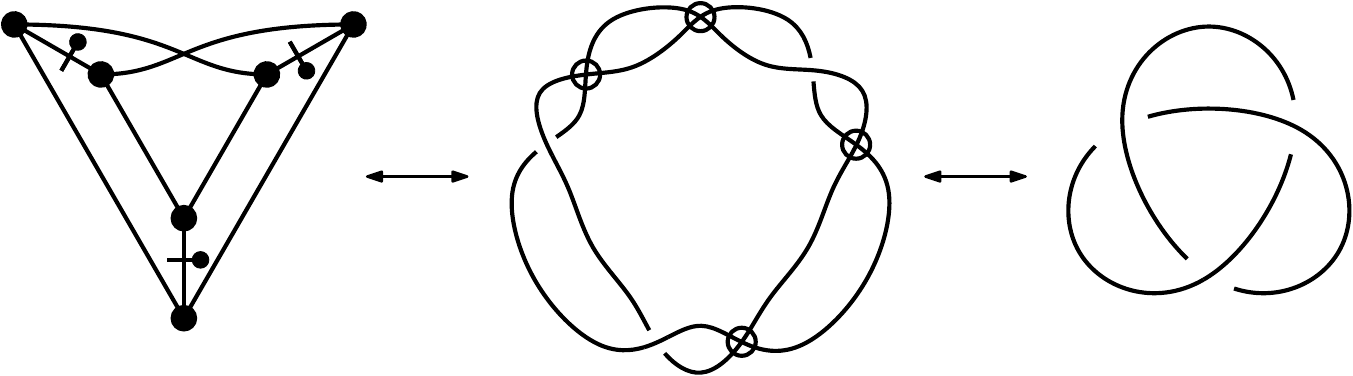}
\end{center}

The set of Reidemeister-like moves on graphs that we consider is given in \Cref{Def:ribbonmoves,Def:graphenemoves}, and the equivalence relation generated by these moves appears to be new to the graph theory literature. These moves allow a graph to be transformed into another distinct graph. For example, we permit the following digon relation
\begin{equation}\label{Eq:g2demo}
	\raisebox{-33pt}{\includegraphics[scale=0.75]{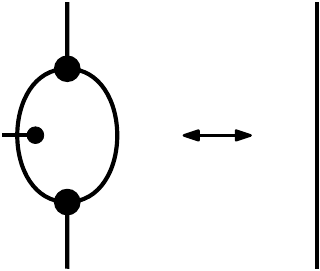}}
\end{equation}
that changes the vertex set of the graph. We refer to an equivalence class of decorated ribbon graphs under these new moves as a \emph{graphene}. This terminology is inspired by the structure of the carbon allotrope graphene, that locally forms a trivalent graph with a perfect matching, as depicted in \Cref{Fig:allotrope}.

\begin{figure}
	\includegraphics[scale=0.5]{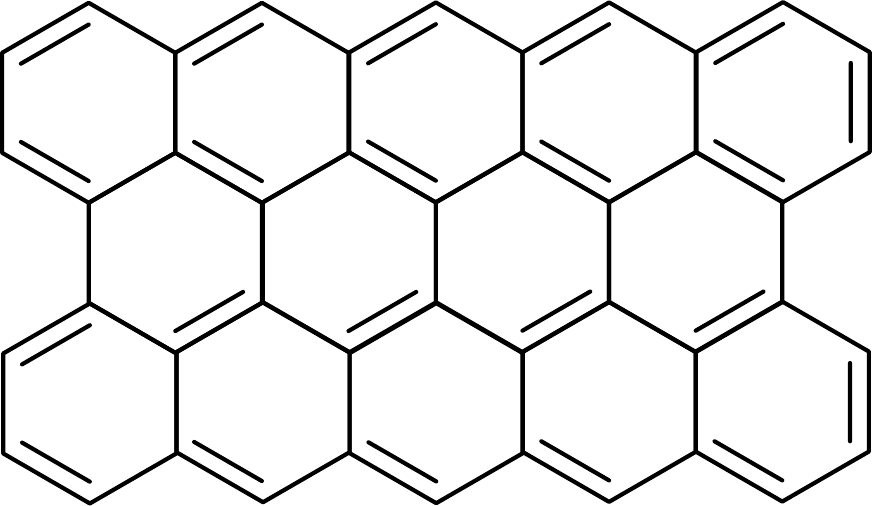}
	\caption{The chemical structure of graphene forms a trivalent graph with a perfect matching given by the double bonds.}
	\label{Fig:allotrope}
\end{figure}

The central result of this paper is that the category of graphenes is isomorphic to that of virtual links.
\begin{theorem*}
	The functor \(\K \) is an isomorphism between the category of graphenes and that of virtual links.
\end{theorem*}
This isomorphism is defined via diagrams, allowing it to be used to transport invariants between the two categories.

As outlined above, a graphene is an equivalence class of decorated ribbon graphs. These decorations include a perfect matching, an orientation of the ribbon graph (as a surface), and further auxiliary data. As such, when pulling back an invariant of virtual links to an invariant of graphenes via \(\K \), the new invariant depends on the decoration in addition to the underlying graph. Nevertheless, in many cases such invariants contain graph-theoretic information that is independent of the decoration; we give examples of such a situation in \Cref{Sec:homologytheories} and \Cref{Sec:embeddings}. It follows that \(\K \) provides an effective way to apply virtual link invariants to the study of graphs. Conversely, \( \K \) allows graph-theoretic constructions to be carried over to knot theory.

\subsection{The Penrose formula and the Jones polynomial}\label{Sec:motivation}
In his investigation of diagrammatic tensor calculus via graphs Penrose discovered a surprising formula that counts the number of Tait colorings (\(3\)-edge colorings) of a planar trivalent graph \cite{Penrose}.  Many modern techniques in graph theory and topology are descendants of this result. In particular, the second author's discovery of the bracket model for the Jones Polynomial was motivated by the Penrose formula \cite{Kauffman1987a}. More recently, the first author generalized the Penrose formula to a \(2\)-isomorphism invariant of trivalent ribbon graphs with perfect matchings \cite{Baldridge2018, Baldridge2}. We now describe how investigating these constructions together with the Penrose formula naturally leads to the functor \(\K\) depicted in \Cref{Eq:kmap}.

Let us outline the Penrose formula. Given a planar trivalent graph, one may resolve its edges with distinct vertices inductively using the formula
\begin{align}
\label{eq:Penformula1}\left[~\raisebox{-0.33\height}{\includegraphics[scale=0.25]{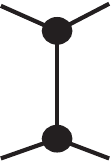}} ~\right] &= \left[\raisebox{-0.33\height}{\includegraphics[scale=0.25]{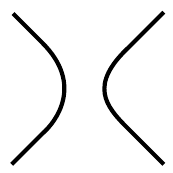}}\right] - \left[\raisebox{-0.33\height}{\includegraphics[scale=0.25]{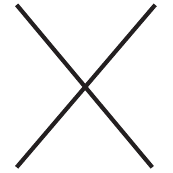}}\right]
\\
\label{eq:Penformula2}\left[\bigcirc \cup D \right] &= 3 \left[D\right],
\end{align}
together with the declaration that the empty graph evaluates to \( 1 \). \Cref{eq:Penformula2} allows for the removal of immersed copies of \( S^1 \), at the expense of multiplying by \( 3\). For any planar trivalent graph, repeated application of \Cref{eq:Penformula1,eq:Penformula2} obtains an integer, equal to the number of Tait colorings of the graph. For example, the number of Tait colorings of the theta graph is 6, and a computation using the Penrose formula produces the same result: \(\left[\raisebox{-0.25\height}{\includegraphics[scale=0.18]{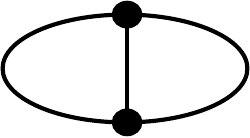}}\right]= 3^2 - 3 = 6\).

The second author noticed that a similar formula holds for links, and that it recovers the Jones polynomial (suitably normalized) \cite{Kauffman1987a}. This {\em Kauffman bracket} resolves crossings of link diagrams instead of edges of graphs, and yields a Laurent polynomial in the variable \(q\):
\begin{align}
\label{eq:bracket-crossing} \left\langle \raisebox{-0.33\height}{\includegraphics[scale=0.25]{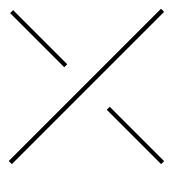}}  \right\rangle &= \left\langle \raisebox{-0.33\height}{\includegraphics[scale=0.25]{A-smoothing.pdf}}  \right\rangle -q \left\langle \raisebox{-0.33\height}{\includegraphics[scale=0.25]{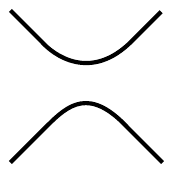}} \right\rangle\\
\label{eq:bracket-disjoint-circ} \left\langle \bigcirc \cup D \right\rangle &= (q^{-1}+q) \left\langle D \right\rangle, 
\end{align}  
together with the declaration that the empty link evaluates to \( 1 \). For example,
\begin{equation*}
	\begin{aligned}
		\left\langle \raisebox{-6pt}{\includegraphics[scale=0.3]{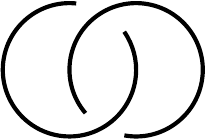}} \right\rangle &= \left\langle \raisebox{-6pt}{\includegraphics[scale=0.3]{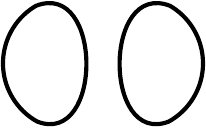}} \right\rangle - q \left\langle \raisebox{-7pt}{\includegraphics[scale=0.3]{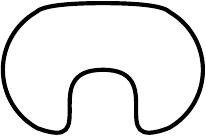}} \right\rangle -q \left\langle \raisebox{-7pt}{\includegraphics[scale=0.3]{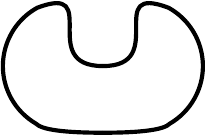}} \right\rangle + q^2 \left\langle \raisebox{-7pt}{\includegraphics[scale=0.3]{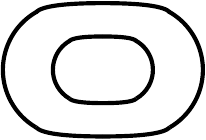}} \right\rangle \\
		&= q^{-2}+1+q^2+q^4
	\end{aligned}
\end{equation*}
and taking an appropriate normalization (determined by the writhe of the diagram) yields the Jones polynomial.

Famously, the Jones polynomial may be \emph{categorified}, i.e.\ lifted to a group-valued invariant, known as \emph{Khovanov homology} \cite{Khovanov1999}. In light of the similarity between the Kauffman bracket and the Penrose formula, it is natural to ask if the latter can be modified to yield a polynomial that can itself be categorified.\footnote{A Penrose polynomial has been defined but does not immediately lend itself to such a categorification \cite{Aigner1997,EllisMonaghan2013APP}.}

One may attempt to modify \Cref{eq:Penformula1,eq:Penformula2} to more closely resemble \Cref{eq:bracket-crossing,eq:bracket-disjoint-circ} by introducing a variable. However, na\"\i ve attempts fail to yield an invariant of the graph: the result depends on the order in which the edges are resolved. The first author overcame this difficulty by considering planar trivalent graphs with a given perfect matching \cite{Baldridge2018}. By resolving only the matched edges, one can define an invariant of the \( (\text{graph},\text{matching}) \) pair, known as the {\em $2$-factor polynomial}. The formula describing this polynomial incorporates aspects of both the Penrose formula and the Kauffman bracket:
\begin{align}
\label{eq:bracket-2-factor}\left\langle ~\raisebox{-0.33\height}{\includegraphics[scale=0.25]{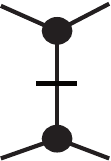}} ~ \right\rangle_2 &= \left\langle \raisebox{-0.33\height}{\includegraphics[scale=0.25]{A-smoothing.pdf}} \right\rangle_2 -q \left\langle \raisebox{-0.33\height}{\includegraphics[scale=0.25]{X-smoothing.pdf}}  \right\rangle_2 \\
\label{eq:disjoint-circ-2-factor}\left\langle \bigcirc \cup D \right\rangle_2 &= (q^{-1}+q) \left\langle D \right\rangle_2
\end{align}
where the crossbar of \Cref{eq:bracket-2-factor} denotes a matched edge.

Notice that evaluating \Cref{eq:bracket-2-factor,eq:disjoint-circ-2-factor} at \( q = 1 \) almost recovers \Cref{eq:Penformula1,eq:Penformula2}: \Cref{eq:disjoint-circ-2-factor} has a coefficient of \(2\) and \Cref{eq:Penformula2} has a coefficient of \(3\). This discrepancy is explained by the fact that the Penrose formula counts the Tait colorings of the argument graph, while the \(2\)-factor polynomial counts only those colorings that assign every matched edge the same color. That is, the \(2\)-factor polynomial counts 2-factors that span the perfect matching \cite{BaldridgeLowranceMcCarty2018}. For example,
\begin{equation*}
	\begin{aligned}
		\left\langle \raisebox{-9pt}{\includegraphics[scale=0.3]{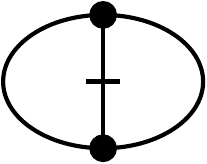}} \right\rangle_2 &= (q^{-1}+q)^2 -q(q^{-1}+q) \\
		&= q^{-2}+1.
	\end{aligned}
\end{equation*}
Evaluating the above polynomial at \( q =1  \) recovers the fact that, up to symmetry, there are exactly two Tait colorings with a fixed color on the matched edge.

The first author defined a homology theory that categorifies the \(2\)-factor polynomial, just as Khovanov homology categorifies the Jones polynomial \cite{Baldridge2018}. Specifically, to a planar trivalent graph with a perfect matching, one may associate a bigraded abelian group, the graded Euler characteristic of which recovers the \(2\)-factor polynomial.

The functor \( \K \) is naturally borne out of the affinity between the \(2\)-factor polynomial of graphs and the Jones polynomial of virtual links established above. Consider the following region of a trivalent graph with a perfect matching: \raisebox{-0.33\height}{\includegraphics[scale=0.4]{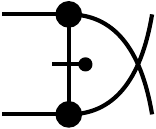}}. Applying \Cref{eq:bracket-2-factor} yields
\begin{equation*}
\left\langle \raisebox{-0.33\height}{\includegraphics[scale=0.4]{pm-edge-with-crossing.pdf}} \right\rangle_2 \ = 
\left\langle \, \raisebox{-0.33\height}{\includegraphics[scale=0.5]{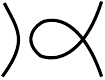}} \, \right\rangle_2 - q \left\langle \, \raisebox{-0.33\height}{\includegraphics[scale=0.5]{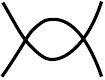}} \, \right\rangle_2 \ = \
 \left\langle \raisebox{-0.33\height}{\includegraphics[scale=0.25]{A-smoothing.pdf}}  \right\rangle_2 -q \left\langle \raisebox{-0.33\height}{\includegraphics[scale=0.25]{B-smoothing.pdf}}  \right\rangle_2 ,\end{equation*}
that corresponds directly to the right hand side of \Cref{eq:bracket-crossing}.

Thus we arrive at the relationship
\begin{equation*}
\raisebox{-0.2\height}{\includegraphics[scale=0.6]{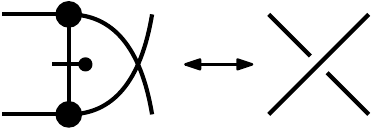}},
\end{equation*}
that is equivalent to \( \K \) as depicted in \Cref{Eq:kmap}. In this paper we elaborate on this relation, proving that \( \K \) yields an isomorphism between a category of equivalence classes of decorated trivalent graphs and the category of virtual links.

\subsection{Outline}\label{Sec:outline}
This paper is organised as follows. In \Cref{Sec:graphenes} we introduce graphenes, certain equivalence classes of decorated ribbon graphs. We then demonstrate that \( \K \) yields an isomorphism between the category of graphenes and the category of virtual links in \Cref{Sec:kmap}. In \Cref{Sec:invariants} we outline new invariants of graphenes that are defined using \( \K \). We discuss various homology theories of virtual links and graphenes in \Cref{Sec:homologytheories}, before demonstrating that one such theory can be used to study strong embeddings of a graph in \Cref{Sec:embeddings}.

\subsection*{Acknowledgements}
Kauffman's work was supported by the Laboratory of Topology and Dynamics, Novosibirsk State University (contract no.\ 14.Y26.31.0025 with the Ministry of Education and Science of the Russian Federation).  All three authors would like to thank Ben McCarty for many helpful conversations and suggestions, and the anonymous referees for their careful reading of the paper.

\section{Graphenes}\label{Sec:graphenes}
In this section we introduce the notion of a graphene, an equivalence class of decorated trivalent ribbon graphs. In \Cref{Subsec:ribbongraphs} we describe ribbon graphs and their diagrams, before defining graphenes in \Cref{Subsec:graphenes}. In these sections a collection of moves on (diagrams of) graphs is introduced; when digesting these moves the reader may benefit from keeping the complementary knot-theoretic moves in mind, as given in \Cref{Fig:crms,Fig:vrms} on \cpageref{Fig:crms}.

\subsection{Ribbon graphs}\label{Subsec:ribbongraphs}
A graphene is an equivalence class of trivalent graphs with extra structure. One of these structures is that of a \emph{ribbon graph}. For a detailed introduction to ribbon graphs see \cite[Section 1.1.4]{Moffat2013}.

\begin{definition}[Ribbon graph]\label{Def:ribbongraph}
A ribbon graph is a graph together with a surface that deformation retracts onto the graph. Given a ribbon graph \( \Gamma \) we denote the abstract graph by \( G_{\Gamma} \), and the surface by \( F_{\Gamma} \). We say that \( G_{\Gamma} \) is the \emph{underlying graph} of \( \Gamma \), and that \( F_{\Gamma} \) is \emph{the surface associated to} \( \Gamma \).

An orientation of a ribbon graph is an orientation of the surface.
\end{definition}

Let \( \Gamma_1 \) and \( \Gamma_2 \) be ribbon graphs. The surfaces \( F_{\Gamma_1} \) and \( F_{\Gamma_2} \) have boundary; let \( \widetilde{F}_{\Gamma_i}\) denote the closed surface obtained by attaching discs to the boundary of \( F_{\Gamma_i} \), for \( i = 1, 2\). The embedding of $G_{\Gamma_i}$ into the surface \( \widetilde{F}_{\Gamma_i} \) is known as a {\em 2-cell embedding}\footnote{Such an embedding is also known as a {\em cellular embedding} or {\em cellular map}.}. We say that \( \Gamma_1\) and \( \Gamma_2 \) are \emph{equivalent} ribbon graphs if there is a homeomorphism \( f : \widetilde{F}_{\Gamma_1} \rightarrow \widetilde{F}_{\Gamma_2} \) such that \( f ( G_{\Gamma_1} ) \) and \( G_{\Gamma_2} \) are isomorphic graphs.

Henceforth all ribbon graphs are assumed to have trivalent underlying graphs and to have oriented associated surfaces (that is, they are orientable with a fixed orientation). All trivalent graphs are assumed to possess at least one perfect matching.

Notice that given a graph \( G \), a cyclic ordering of the edges at every vertex determines a ribbon graph, \( \Gamma \), with underlying graph \( G \). This ribbon graph is obtained by taking a disc for every vertex of \( G \), and attaching bands as prescribed by the edges and their cyclic ordering. Half twists may be added to the bands, provided that the resulting surface satisfies the orientability condition. It follows that the vertices (edges) of \( G \) are in bijection with the discs (bands) of \( \Gamma \), and we shall not distinguish between them, referring to \emph{vertices} and \emph{edges} of \( \Gamma \).

Two distinct ribbon graphs may have the same underlying abstract graph; an example is given in \Cref{Fig:ribbons}. These ribbon graphs are distinguished by the number of boundary components of their associated surfaces.

\begin{figure}
\includegraphics[scale=0.75]{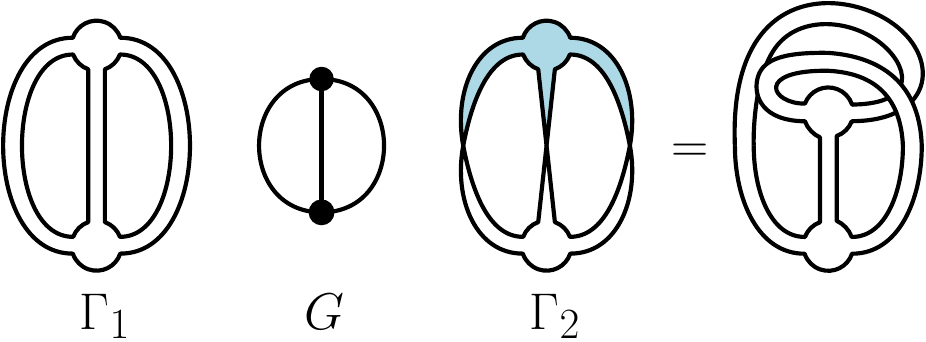}
\caption{Distinct ribbon graphs \( \Gamma_1 \) and \( \Gamma_2 \) with the same underlying graph \( G \).}
\label{Fig:ribbons}
\end{figure}

We represent ribbon graphs by the following diagrams.
\begin{definition}[Ribbon diagram]\label{Def:ribbondiagram}
A \emph{ribbon diagram} is a graph drawn in the plane (with possible intersections between its edges), with vertices decorated as either solid, \raisebox{-6pt}{\includegraphics{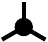}}, or hollow, \raisebox{-6pt}{\includegraphics{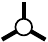}}.
A cyclic ordering of the edges at a vertex is given implicitly by such a diagram i.e.\ it is given by their ordering in the plane.
\end{definition} 

The hollow vertex is a notational abbreviation for a solid vertex with a twist; see the $M5$ move of \Cref{Def:ribbonmoves}. An example of a ribbon diagram is given in \Cref{Fig:ribbondiagram}. 

\begin{figure}
\includegraphics{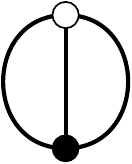}
\caption{A ribbon diagram.}
\label{Fig:ribbondiagram}
\end{figure}

Given a ribbon diagram we can recover a ribbon graph as follows. Recall that ribbon graphs are assumed to be oriented; the outward normal side of an oriented surface is depicted unshaded, and the inward normal side shaded. That is, the two sides of the oriented surface are distinguished by the shading.
\begin{enumerate}
	\item Replace the solid vertices of the diagram with the following surface component:
	\begin{center}
	\includegraphics[scale=0.75]{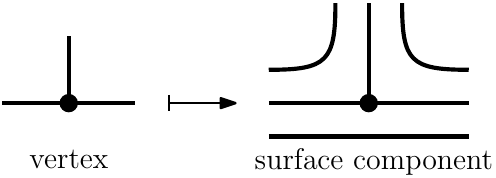}
	\end{center}
	\item Replace the hollow vertices of the diagram with the following surface component:
	\begin{center}
	\includegraphics[scale=0.75]{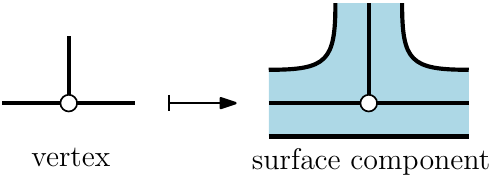}
	\end{center}
	\item Glue the surface components together as dictated by the incidence matrix of the graph depicted in the ribbon diagram, adding any half-twists necessary to produce a consistently oriented surface.
\end{enumerate}
If \( \Gamma \) is obtained from a ribbon diagram, \( D \), in this manner we say that \( D \) \emph{represents} \( \Gamma \). For example, the ribbon diagram given in \Cref{Fig:ribbondiagram} represents the ribbon graph \( \Gamma_2 \) in \Cref{Fig:ribbons}.

A ribbon diagram \( D \) is a decorated plane drawing of a graph. If \( D \) is such a plane drawing of a graph \( G \), and \( D \) represents the ribbon graph \( \Gamma \), then \( G \) is the underlying graph of \( \Gamma \), by construction.

\begin{proposition}\label{Prop:ribbonrep}
Every ribbon graph is represented by ribbon diagram.
\end{proposition}

\begin{proof}
Let \( \Gamma \) be a ribbon graph, and consider an arbitrary embedding of \( F_{\Gamma} \) in \( \mathbb{R}^3 \). Denote the co-ordinates of \( \mathbb{R}^3 \) as \( \left( x , y , z \right) \). Isotope the embedding so that projection onto the \( z = 0 \) plane restricts to a regular projection on \( G_{\Gamma} \), injective at its vertices. The image of this regular projection is a plane drawing of \( G_{\Gamma} \). The orientation of \( F_{\Gamma} \) yields a normal vector at every point of \( G_{\Gamma} \), and as the projection is injective at vertices, there is a well defined normal vector to the vertices of the plane drawing. Decorate a vertex of the plane drawing as solid if the \(z\)-component of its normal vector is positive, and hollow if it is negative (the embedding may always be isotoped so that the \(z\)-component is non-zero). This procedure yields a ribbon diagram representing \( \Gamma \).
\end{proof}

In order for ribbon diagrams to faithfully represent ribbon graphs up to equivalence, we must introduce the following diagrammatic moves. These moves allow one to move between two distinct ribbon diagrams of the same ribbon graph i.e.\ between two distinct planar representations.

\begin{definition}[Ribbon  moves]\label{Def:ribbonmoves}
The following moves on ribbon diagrams are known as the \emph{ribbon  moves}:
\begin{center}
\includegraphics[scale=0.7]{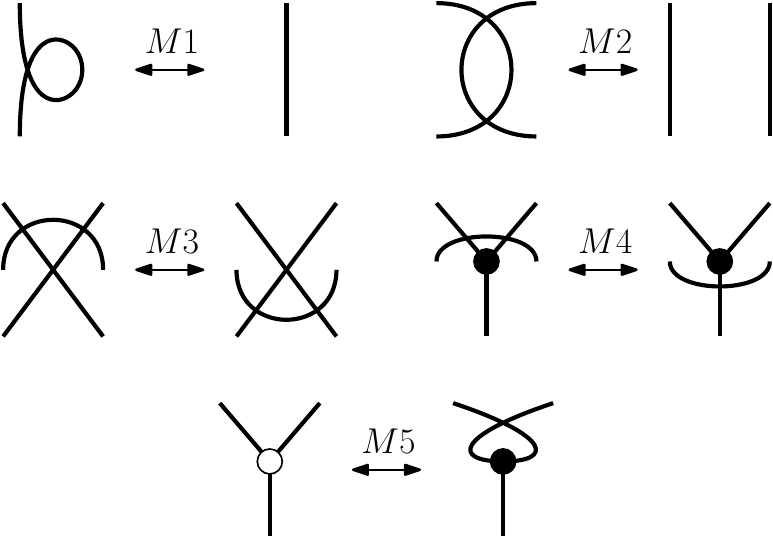}
\end{center}
where the vertices in \( M4 \) are either both solid, or both hollow.
\end{definition}

\begin{remark}
	Throughout this paper we could work with only solid vertices, using the move \( M5 \) to convert every hollow vertex to solid. However, using hollow and solid vertices has the benefit of greatly simplifying ribbon diagrams. For example, the leftmost diagram of \Cref{Fig:matcheddiagrams} and decorated Franklin graph of \Cref{Ex:franklin} become unwieldy if the hollow vertices are converted.
\end{remark}

A useful consequence of the ribbon moves allows edges to be transformed arbitrarily, provided their endpoints remain fixed.
\begin{lemma}[Detour move]\label{Lem:detour}
The ribbon moves induce the following move on ribbon diagrams, known as the \emph{detour move}:
\begin{center}
\includegraphics[scale=0.65]{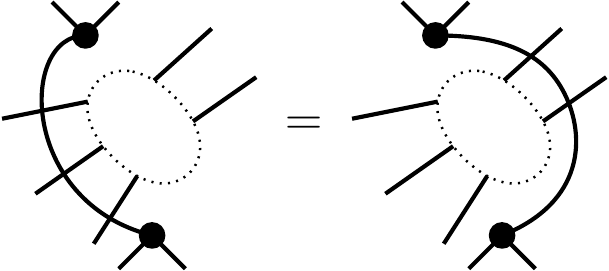}
\end{center}
\end{lemma}

Introducing the ribbon moves allows us to convert one diagram of a ribbon graph to another.
\begin{theorem}\label{Thm:ribbons}
Two ribbon diagrams represent the equivalent ribbon graphs if and only if they are related by a finite sequence of ribbon moves and planar isotopy.
\end{theorem}

\begin{proof}
It is clear from the construction of the ribbon moves that if two ribbon diagrams are related by such a sequence, then they represent the equivalent ribbon graphs.

To deduce the converse, recall that a ribbon graph is equivalent to the following data: a vertex set, an incidence matrix, a cyclic ordering of the edges at each vertex, and an orientation of the surface defined by the cyclic ordering\footnote{This ordering data is also known as a {\em rotational system}.}.

On the other hand, a ribbon diagram is equivalent to the following data: a vertex set, a choice of edges in the plane realizing an incidence matrix, a cyclic ordering of the edges at each vertex, and vertex decorations (either solid or hollow).

The detour move of \Cref{Lem:detour} eliminates the dependence on the choice of realization of the incidence matrix. That is, any choice of path may be transformed into another while keeping the endpoints fixed. It follows that the equivalence class of ribbon diagrams up to the ribbon moves depends only on the incidence matrix.

Further, to explicitly write down the cyclic order of edges at a vertex of a ribbon graph one must choose to view the surface from the outward normal side, or the inward normal side of the surface. Changing the side reverses the ordering. This is captured exactly by the move \( M5 \).
\end{proof}
As a consequence of \Cref{Thm:ribbons}, we may equivalently define a ribbon graph as an equivalence class of ribbon diagrams, up to the ribbon moves.

\subsection{Graphenes}\label{Subsec:graphenes}
To define a graphene we require \emph{directed perfect matchings} of a trivalent graphs, in addition to the ribbon structure.

\begin{definition}[Perfect matching]\label{Def:pm}
A \emph{perfect matching} of a graph is a subset of the edges of the graph, such that each vertex is incident to exactly one edge in the subset. The subset is known as the \emph{matching set}, and edges in the matching set are known as \emph{matched edges}.
\end{definition}
The term {\em matching} is used in graph theory for any subset of the edge set of graph; the term \emph{perfect} here refers to the fact that every vertex is incident to exactly one matched edge.

\begin{definition}[Directed perfect matching]\label{Def:dpm}
A \emph{directed perfect matching} is a perfect matching with two extra decorations on each matched edge: a {\em direction} and a {\em sign}. A {\em direction} distinguishes one end of the matched edges by placing a dot, as in \raisebox{-6pt}{\includegraphics{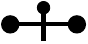}}, and similarly for other vertex decorations. A positive edge is denoted by a solid line and a negative edge is denoted by a dotted line.
\end{definition}

The direction of a perfect matching breaks the symmetry of the matched edges under \( 180 \) degree rotation.

\begin{remark} One could use an arrowhead \raisebox{-2pt}{\includegraphics{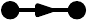}} instead of \raisebox{-6pt}{\includegraphics{matchededge.pdf}} to denote direction.  However, such arrowheads are easily confused with orientations of edges.
\end{remark}

We shall study equivalence classes of ribbon graphs with directed perfect matchings of their underlying graphs. Henceforth we shall refer to these objects as \emph{matched graphs}, and represent them using the following diagrams.
\begin{definition}[Matched diagram]\label{Def:md}
A \emph{matched diagram} is a ribbon diagram with a directed perfect matching of the graph. That is, it is a plane drawing of a graph with a directed perfect matching and the vertex decorations of \Cref{Def:ribbondiagram}. A positive matched edge is denoted by \raisebox{-6pt}{\includegraphics{matchededge.pdf}} and a negative matched edge by \raisebox{-6pt}{\includegraphics{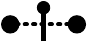}}.
\end{definition}

\begin{figure}
\includegraphics[scale=0.65]{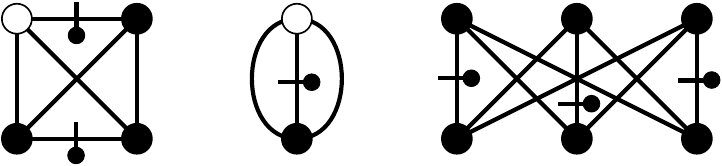}
\caption{Matched diagrams.}
\label{Fig:matcheddiagrams}
\end{figure}
Examples of matched diagrams are given in \Cref{Fig:matcheddiagrams}. We now introduce a collection of moves on matched diagrams. The direct correspondence between these moves and the Reidemeister moves of knot theory (as depicted in \Cref{Fig:crms}) shall be made explicit in \Cref{Sec:kmap}.

\begin{definition}[Graphene moves]\label{Def:graphenemoves}
The moves on matched diagrams depicted in \Cref{Fig:non-ribbonmoves} are known as the \emph{non-ribbon moves}. Taking the union of the ribbon and non-ribbon moves yields the \emph{graphene moves}.
\end{definition}

\begin{figure}
	\includegraphics[scale=0.65]{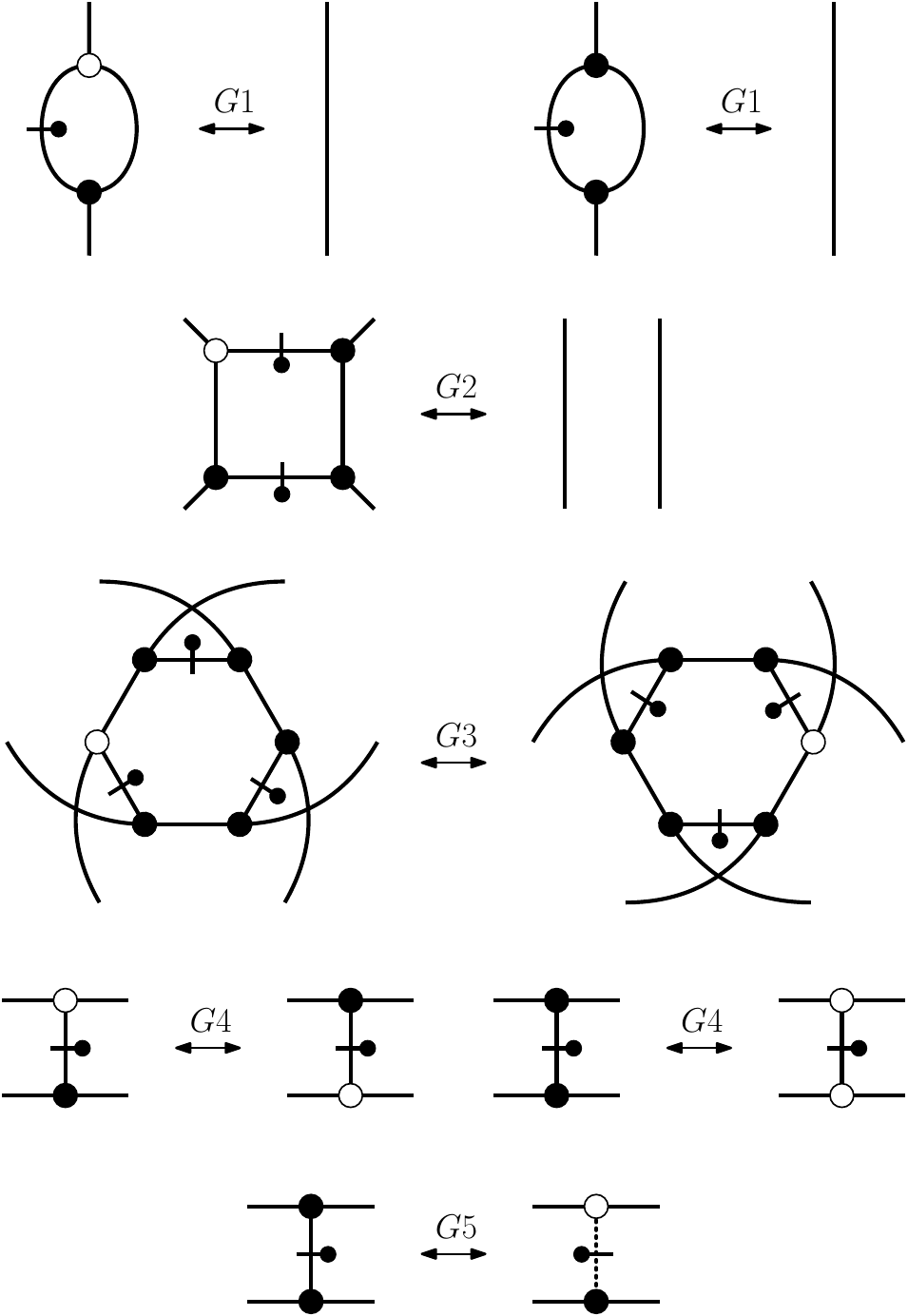}
	\caption{The non-ribbon moves.}
	\label{Fig:non-ribbonmoves}
\end{figure}

That the moves given in \Cref{Fig:non-ribbonmoves} are sufficient for our purposes follows from a result of Polyak \cite{Polyak2010} and the work of \Cref{Sec:kmap}.

The reader who is familiar with the Reidemeister moves is encouraged to apply the replacement depicted in \Cref{Eq:kmap} to (both sides of) the rightmost \( G1 \) move, and to satisfy themselves that the first Reidemeister move is obtained.

As observed in \Cref{Thm:ribbons}, the ribbon moves preserve the ribbon graph corresponding to a matched diagram. In contrast, the non-ribbon moves may alter the ribbon graph. The moves \( G1 \), \( G2 \), and \( G3 \) provide bigon, square, and hexagon relations, respectively, all of which change the underlying graph of a ribbon graph. The \( G4 \) move and the $G5$ move preserve the underlying graph, but alter the associated surface.

We are now in a position to define a graphene.
\begin{definition}[Graphene]\label{Def:graphene}
A \emph{graphene} is an equivalence class of matched diagrams, up to the graphene moves. Two matched diagrams represent the same graphene if they are related by a finite sequence of graphene moves and planar isotopies.
\end{definition}
If a matched diagram \( D \) represents the graphene \( \mathcal{G} \) we say that \emph{\( D \) is a diagram of \( \mathcal{G} \)}; two diagrams which represent the same graphene are said to be \emph{equivalent}.

In light of \Cref{Thm:ribbons}, we may equivalently define graphenes as equivalence classes of ribbon graphs with directed perfect matchings, up to the non-ribbon moves. We say that a ribbon graph \( \Gamma \) \emph{represents} a graphene \( \mathcal{G} \) if there is a matched diagram of \( \mathcal{G} \) that represents \( \Gamma \) as a ribbon diagram (that is, forgetting the directed perfect matching). For reasons of exposition it is useful to keep both definitions in mind.

The graphene moves allow matched diagrams to be manipulated in many ways. Given a diagram of a graphene we may use the ribbon move \( M5 \) to produce an equivalent diagram whose vertices are all solid, or all hollow, at the cost of possibly introducing intersections between edges.

The move \( G4 \) has the following action on the surface associated to the ribbon graph defined by a matched diagram: remove a portion of the surface homeomorphic to that on the left of \Cref{Fig:perfectmatchingmove}, and replace it with a portion homeomorphic to that on the right. (There is also a version in which the central band has a half-twist.) This move may be motivated as follows.

\begin{figure}
\includegraphics[scale=0.65]{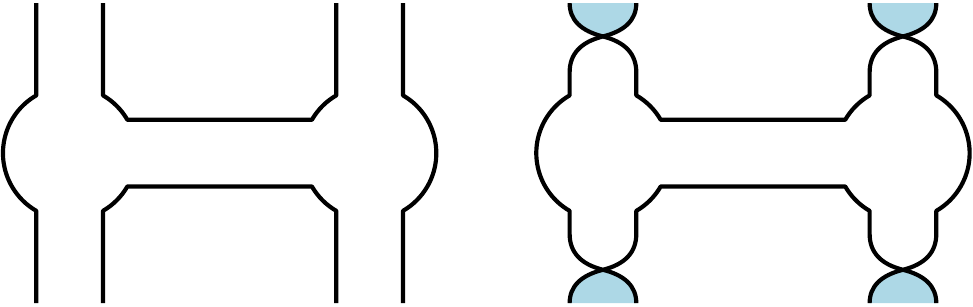}
\caption{The affect of \( G4 \)  on a ribbon surface.}
\label{Fig:perfectmatchingmove}
\end{figure}

Two distinct ribbon graphs may possess the same underlying graph. A graphene is an equivalence class of ribbon graphs, and including \( G4 \) can cause distinct ribbon graphs to represent the same graphene. As one of our main motivations is to apply the graphene construction to the study of trivalent graphs, reducing the dependence of the construction on the ribbon graph is desirable.	

Another datum forming a graphene is a choice of orientation of the associated surface of a ribbon graph. Including the move \( G4 \)  eliminates the dependence on this choice.

\begin{lemma}\label{Lem:orientationflip}
Let \( \Gamma_1 \) be a ribbon graph, and \( \Gamma_2 \) the ribbon graph obtained by reversing orientation of \( F_{\Gamma_1} \). Then \( \Gamma_1 \) and \( \Gamma_2 \) represent the same graphene.
\end{lemma}

\begin{proof}
Recall that solid (hollow) vertices represent the outward (inward) normal side of \( F_{\Gamma_1} \). Therefore reversing the orientation transforms solid vertices to hollow, and vice versa. As \( \Gamma_1 \) and \( \Gamma_2 \) differ only in the orientation of the surface \( F_{\Gamma_1} \), there exists a ribbon diagram, \( D_1 \), of \( \Gamma_1 \), and a ribbon diagram, \( D_2 \), of \( \Gamma_2 \), such that \( D_2\) is obtained from \( D_1 \) by converting hollow vertices to solid, and vice versa. Such a transformation is achievable via \( G4 \) , and it follows that \( D_1 \) and \( D_2 \) represent the same graphene, so that \( \Gamma_1 \) and \( \Gamma_2 \) do also.
\end{proof}

Other useful consequences of the graphene moves are as follows.
\begin{lemma}\label{Lem:graphenezmove}
We have
\begin{equation*}
	\raisebox{-21pt}{\includegraphics[scale=0.75]{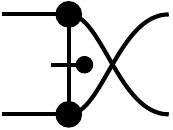}}\quad = \quad\raisebox{-21pt}{\includegraphics[scale=0.75]{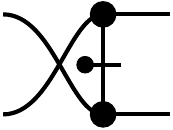}}
\end{equation*}
The above equation holds for any vertex decoration and direction of the matched edges.
\end{lemma}

\begin{lemma}\label{Lem:rotate}
We have
\begin{equation*}
	\raisebox{-21pt}{\includegraphics[scale=0.75]{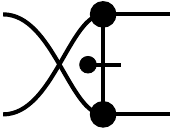}}\quad = \quad\raisebox{-21pt}{\includegraphics[scale=0.75]{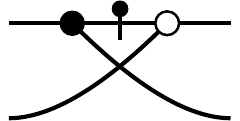}}\quad = \quad\raisebox{-21pt}{\includegraphics[scale=0.75]{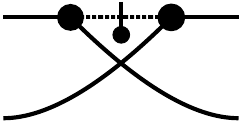}} 
\end{equation*}
\end{lemma}
The proofs of \Cref{Lem:graphenezmove,Lem:rotate} are exercises in applying the graphene moves.

\section{The correspondence between graphenes and virtual links}\label{Sec:kmap}
In this section we define a functor \( \K \) from the category of graphenes to that of virtual links, and demonstrate that it is an isomorphism. In \Cref{Subsec:vkt} we give a brief review of virtual knot theory, before moving on to the definition of \( \K \) itself in \Cref{Subsec:kmap}. We shall use this relationship in \Cref{Sec:invariants,Sec:embeddings} to construct invariants of graphenes.

\subsection{Virtual knot theory}\label{Subsec:vkt}
Classical knot theory is the study of embeddings of disjoint unions of \( S^1 \) in \( S^3 \). Let \( \Sigma_g \) denote a closed orientable surface of genus \( g \); virtual knot theory is the study of embeddings of disjoint unions of \( S^1 \) in \(3\)-manifolds of the form \( \Sigma_g \times [ 0 , 1 ] \), up to certain equivalence. The study of such links was initiated by Kauffman \cite{Kauffman1998}. Unlike links in many other \(3\)-manifolds, virtual links have a diagrammatic theory, akin to that of classical links. As such, working with virtual links is in some cases similar to working with classical links. Indeed, classical knot theory is a proper subset of virtual knot theory, and many constructions in classical knot theory generalize to the virtual domain.

In this section we give a brief overview of virtual knot theory. First we describe the diagrammatic theory of virtual links, before outlining the corresponding topological interpretation. For more thorough introductions to virtual knot theory, see \cite{Kauffman1998,Manturov2012}.

\subsubsection{Virtual links via diagrams}\label{Subsec:diagrams}
We may represent virtual links using diagrams in the plane, with a new crossing decoration.
\begin{definition}\label{Def:virtualdiagram}
	A \emph{virtual link diagram} is a \( 4 \)-valent planar graph, the vertices of which are decorated with either the classical overcrossing and undercrossing decorations, or a new decoration, \raisebox{-3pt}{\includegraphics[scale=0.35]{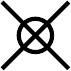}}, known as a \emph{virtual crossing}.
\end{definition}
The virtual crossing is naturally interpreted graph-theoretically: classical crossings are decorated nodes of a \(4\)-valent graph, whereas virtual crossings are simply intersections between its edges.

Examples of virtual link diagrams are given in \Cref{Fig:detourmove2}. There is a set of moves on virtual diagrams that generalizes the Reidemeister moves of classical knot theory.

\begin{definition}\label{Def:vrms}
The \emph{virtual Reidemeister moves} consist of those of classical knot theory, depicted in \Cref{Fig:crms}, together with four new moves involving virtual crossings, depicted in \Cref{Fig:vrms}.
\end{definition}

\begin{figure}
\includegraphics[scale=0.65]{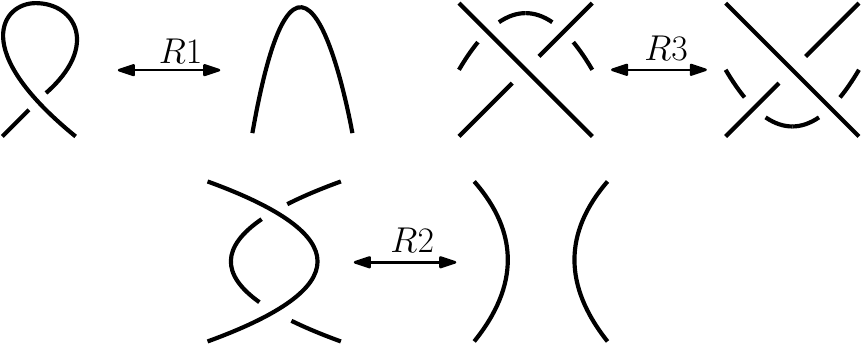}
\caption{The classical Reidemeister moves.}
\label{Fig:crms}
\end{figure}

\begin{figure}
\includegraphics[scale=0.65]{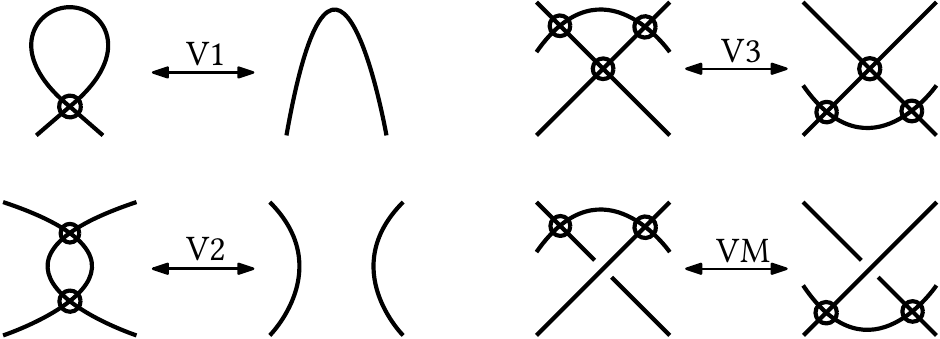}
\caption{The purely virtual Reidemeister moves.}
\label{Fig:vrms}
\end{figure}

\begin{definition}\label{Def:virtuallink}
	A \emph{virtual link} is an equivalence class of virtual link diagrams, up to the virtual Reidemeister moves. Two virtual link diagrams are \emph{equivalent} if they are related by a finite sequence of virtual Reidemeister moves and planar isotopy.
\end{definition}

The virtual Reidemeister moves are designed so that all of the moves involving virtual crossings are instances of one move, called the \emph{detour move}, depicted in \Cref{Fig:detourmove}. In it, a line segment is excised and a new connection made between the resulting endpoints. Any intersections introduced in this manner are virtual crossings. An example of this move in practice is given in \Cref{Fig:detourmove2}. 

\begin{figure}
\includegraphics[scale=0.7]{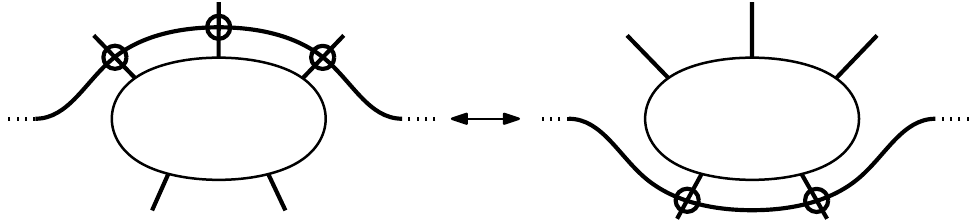}
\caption{The detour move.}
\label{Fig:detourmove}
\end{figure}

\begin{figure}
\includegraphics[scale=0.7]{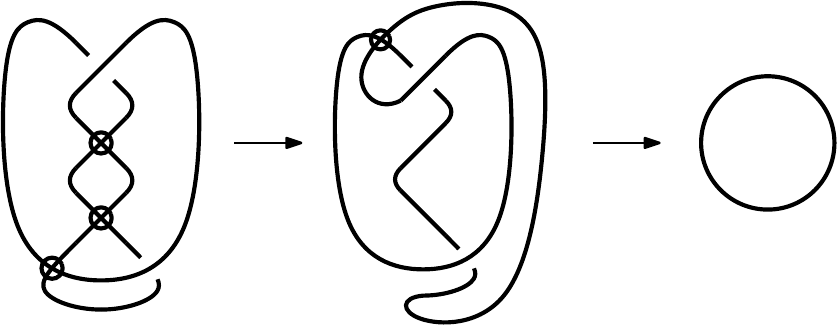}
\caption{The detour move in practice.}
\label{Fig:detourmove2}
\end{figure}

A classical link diagram is simply a virtual link diagram without virtual crossings. One may ask if introducing the virtual Reidemeister moves given in \Cref{Fig:vrms} causes two inequivalent classical diagrams to become equivalent. This was answered in the negative by Goussarov, Polyak, Viro, Kauffman, and Kuperberg.

\begin{theorem}[\cite{Goussarov2000, Kuperberg2002, Kauffman1998}]\label{Thm:GPV}
Let \( D_1 \) and \( D_2 \) be classical link diagrams. Then \( D_1 \) and \( D_2 \) are related by the virtual Reidemeister moves if and only if they are related by the classical Reidemeister moves.
\end{theorem}
It follows that classical knot theory is a proper subset of virtual knot theory. There exist virtual link diagrams which are not equivalent to a classical link diagram; an example is given in the bottom left of \Cref{Fig:thickenedsurfaces}. For an in-depth treatment of the diagrammatic theory of virtual links see \cite{Kauffman1998}.

\subsubsection{Virtual links as links in thickened surfaces}\label{sec:virt-links-in-thick-surfaces}
Virtual links have an alternative topological realization. As before, let \( \Sigma_g \) denote a closed orientable surface of genus \( g \). A \emph{thickened surface} is a \(3\)-manifold of the form \( \Sigma_g \times [ 0 , 1 ] \). Virtual links may be realized as equivalence classes of embeddings of (disjoint unions of) \( S^1 \) in \( \Sigma_g \times [ 0 , 1 ] \).

\begin{figure}
\includegraphics[scale=0.65]{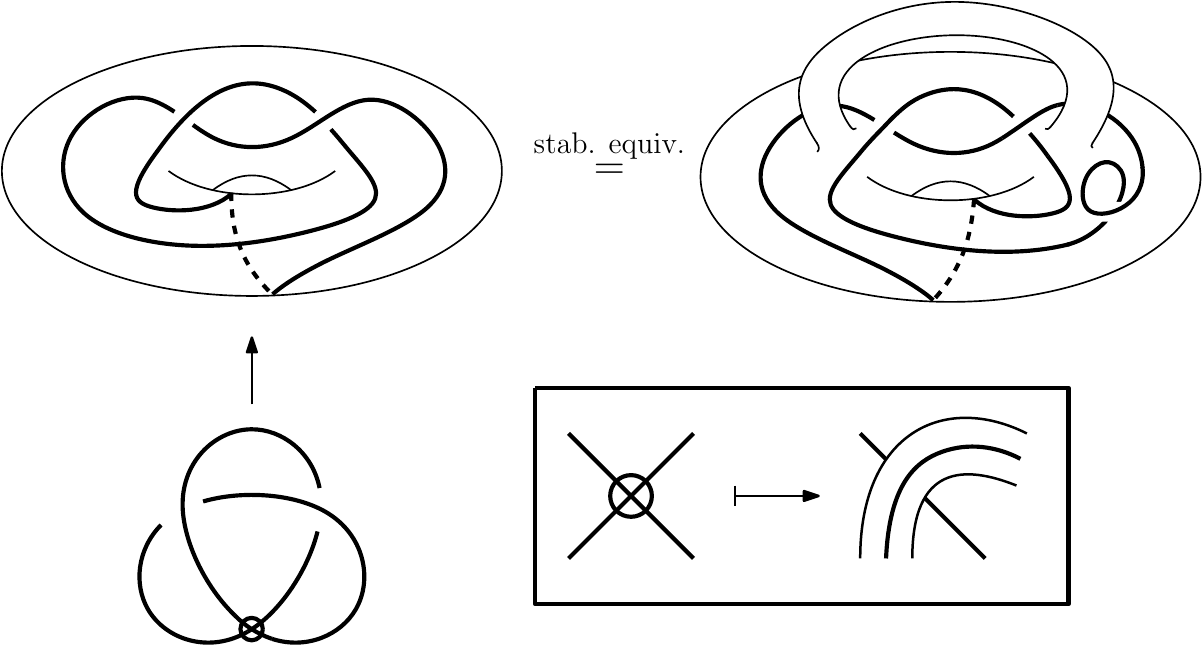}
\caption{Virtual links represent links in thickened surfaces.}
\label{Fig:thickenedsurfaces}
\end{figure}

Specifically, let \( D \) be a virtual link diagram, viewed as a diagram on the surface of a \(2\)-sphere. At every virtual crossing of \( D \) make the replacement depicted in the bottom-right of \Cref{Fig:thickenedsurfaces}. That is, add a handle to the \(2\)-sphere, and place one arc across it. The result is a link diagram on a (closed orientable) surface. Such a diagram represents a link in a thickened surface. An example of this situation is given on the left of \Cref{Fig:thickenedsurfaces}.

Given a virtual link diagram \( D \), denote by \( \overline{D} \) the diagram on a surface resulting from the process described above. Two such diagrams \( \overline{D}_1 \) and \( \overline{D}_2 \) are \emph{stably equivalent} if one can be converted to the other via a finite sequence of classical Reidemeister moves (on \( \Sigma_g \)), self-diffeomorphism of the ambient surface, and addition or removal of empty handles. That is, one is permitted to add or remove handles which do not contain a part of the diagram; for more details see \cite{Carter2000}. An example of a pair of stably equivalent diagrams is given in \Cref{Fig:thickenedsurfaces}.

The virtual Reidemeister moves capture stable equivalence.
\begin{theorem}[\cite{Kauffman1998,Carter2000}.]\label{Thm:stable}
Let \( D_1 \) and \( D_2 \) be virtual link diagrams. Then \( D_1 \) and \( D_2 \) are equivalent if and only if \( \overline{D}_1 \) and \( \overline{D}_2 \) are stably equivalent.
\end{theorem}
This result demonstrates that the diagrammatic theory described in \Cref{Subsec:diagrams} describes a topological theory of (equivalences classes of) links in thickened surfaces. As mentioned above, there exist virtual link diagrams which are not equivalent to classical link diagrams. Therefore for such links, the process outlined above shall always yield a surface of non-zero genus.

\subsection{The \( \mathbb{K} \) functor}\label{Subsec:kmap}
Let \( D( \mathbb{G}) \) denote the category whose objects are matched diagrams, and in which there is a morphism between two objects if and only if the associated matched diagrams are related by a finite sequence of graphene moves. Denote by \( \mathbb{G} \) the category whose objects are graphenes, and possessing no non-identity morphisms.

Similarly let \( D( \mathbb{L}) \) denote the category whose objects are virtual link diagrams, and in which there is a morphism between two objects if and only if the associated virtual link diagrams are related by a finite sequence of virtual Reidemeister moves. Denote by \( \mathbb{L} \) the category whose objects are virtual links, and possessing no non-identity morphisms.

In this section we establish an isomorphism between \( D(\mathbb{G}) \) and \( D(\mathbb{L}) \), that descends to an isomorphism between \( \mathbb{G} \) and \( \mathbb{L} \). First we define a functor on a restricted class of matched diagrams, before extending it to all such diagrams. We then show that this functor interacts appropriately with the graphene moves and virtual Reidemeister moves, so that it defines a functor from \( \mathbb{G} \) to \( \mathbb{L} \). Finally, we show that this functor possesses an inverse.

\begin{definition}\label{Def:kmap1}
Let \( X  \) be the subcategory of \( D (\mathbb{G}) \) consisting of matched diagrams such that (i) each vertex is solid and (ii) the interior of each matched edge does not intersect itself or any other edge interior. The morphisms of \( X \) are those finite sequences of graphene moves involving only matched diagrams of this form. Define the functor \( \K \) from \( X \) to \( D (\mathbb{L}) \) as follows. Given a matched diagram in \( X \) make the following replacement at matched edges, as dictated by the direction:
\begin{center}
	\includegraphics[scale=0.75]{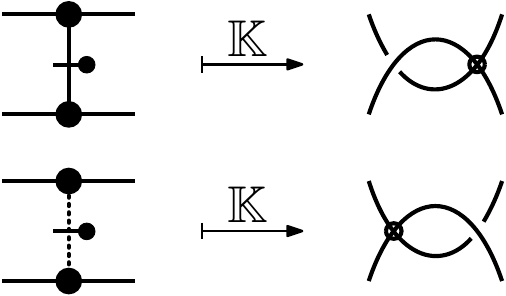}
\end{center}
Intersections between two non-matched edges are replaced with virtual crossings \raisebox{-3pt}{\includegraphics[scale=0.35]{virtualcrossing.pdf}}.

A morphism in \(X\) is a finite sequence of graphene moves. The functor \( \K \) acts on such a morphism by applying the above replacement at every stage.
\end{definition}

We extend this definition to all matched diagrams via the following lemma, that amounts to repeated use of the $M5$ move and the detour move.

\begin{lemma}\label{lemma:extend_K}
The functor $\mathbb{K}$ extends via the ribbon moves to a functor from \( D (\mathbb{G}) \) to \( D (\mathbb{L}) \).
\end{lemma}

\begin{proof}
Given a matched diagram convert all hollow vertices into solid vertices via the ribbon move $M5$. If a matched edge has self-intersections, remove them using the moves \( M1 \) and \( M2 \) of \Cref{Def:ribbonmoves}. If a matched edge has intersections with non-matched edges, use moves \( M3 \) and \( M4 \) to remove them.  Triple points and higher intersections may be arbitrarily separated into double points. The resulting matched diagram will be an element of \( X \), so that the functor \( K \) of \Cref{Def:kmap1} may be applied to it.
\end{proof}

By an abuse of notation we shall henceforth use \( \K \) to denote the extension of the functor defined in \Cref{Def:kmap1} to all matched diagrams obtained via \Cref{lemma:extend_K}.

Although its definition is in terms of diagrams, \( \mathbb{K} \) descends to a functor on graphenes.
\begin{proposition}\label{Prop:welldefined1}
Let \( D_1 \) and \(D_2\) be matched diagrams. If \( D_1 \) is related to \( D_2 \) via graphene moves, then \( \mathbb{K} ( D_1 ) \) is related to \( \mathbb{K} ( D_2 ) \) via virtual Reidemeister moves.
\end{proposition}

\begin{proof}
Let \( D_1 \) and \( D_2 \) be matched diagrams. We are required to show that if \( D_1 \) is related to \( D_2 \) via a finite sequence of graphene moves and planar isotopies, then \( \mathbb{K} ( D_1 ) \) is related to \( \mathbb{K} ( D_2 ) \) via a finite sequence of virtual Reidemeister moves and planar isotopies. It suffices to verify this for sequences of length one: the proposition then follows by transitivity.

The ribbon moves \( M1 \), \( M2 \), and \( M3 \) of \Cref{Def:ribbonmoves} are sent to the moves \( V1 \), \( V2 \), and \( V3 \) of \Cref{Fig:vrms}; to see this, apply the \(\K \) functor to both sides of each move.

The move \( M4 \) may be replicated on virtual link diagrams by a combination of the moves \( V2 \), \( V3 \) and \( VM \). For example,
\begin{center}
\includegraphics[scale=0.75]{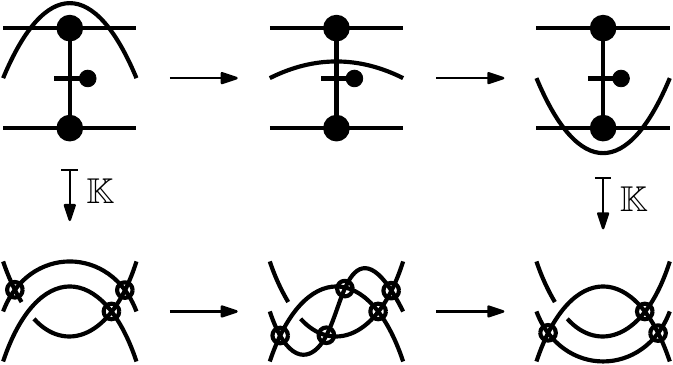}
\end{center}
Other vertex and direction configurations follow similarly.

It remains to verify the proposition for the non-ribbon moves. By using the $M5$ move, applying $\K$, and performing one detour move, we obtain the following
\begin{equation}\label{eq:K-map-one-hollow-vertex}
		\includegraphics[scale=0.75]{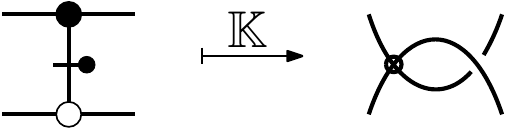}.
\end{equation}
Using this identity one can quickly verify that if \( D_1 \) and \( D_2 \) are matched diagrams related by either form of the \( G4 \) move, then \( \K ( D_1 ) \) and \( \K ( D_1 ) \) represent the same virtual link.

Rotating both diagrams of \Cref{eq:K-map-one-hollow-vertex} by $180^\circ$ and comparing the result to the second map in \Cref{Def:kmap1}, one obtains an equivalent result holds for the $G5$ move.

The moves \( G1 \), \( G2 \) and \( G3 \) in \Cref{Def:graphenemoves} correspond to the classical Reidemeister  moves $R1$, $R2$, and $R3$ of \Cref{Fig:crms} respectively. We verify that \( \K \) is well defined with respect to the second \( G 2 \) move in \Cref{Def:graphenemoves}, leaving the other cases to the reader. We have
\begin{center}
\includegraphics[scale=0.65]{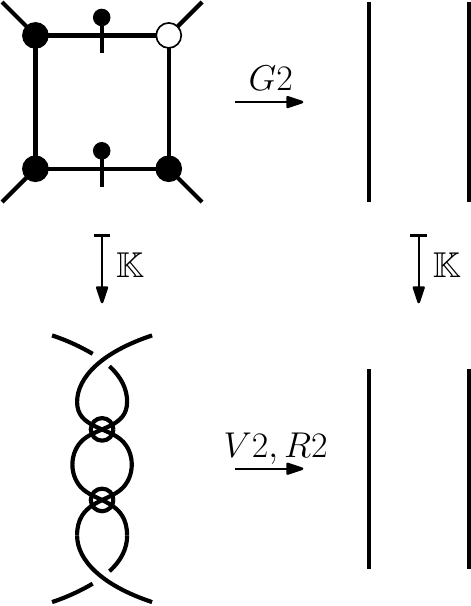}
\end{center}
\end{proof}
In light of \Cref{Prop:welldefined1}, we abuse notation to denote by \( \mathbb{K} \) the functor from \( \mathbb{G} \) to \( \mathbb{L} \) induced by \Cref{Def:kmap1}.

\begin{example}\label{Ex:petersen}
The following graphene, with underlying graph the Petersen graph, is mapped to a \(2\)-component link under \( \mathbb{K} \) 
\begin{center}
\includegraphics[scale=0.5]{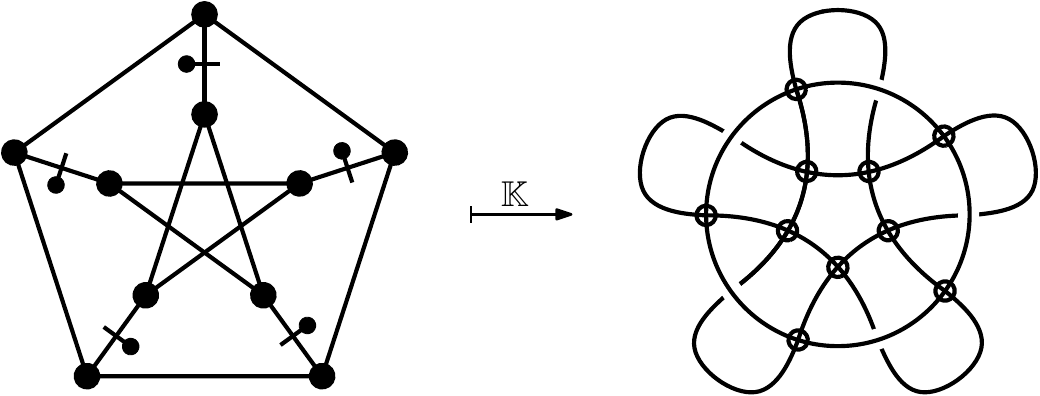}
\end{center}
\end{example}

Our next task is to verify that \( \K \) has an inverse $\Kinv$.
\begin{definition}\label{Def:kmap2}
Define a functor from \( D ( \mathbb{L} ) \) to \( D ( \mathbb{G} ) \) as follows. Given a virtual link diagram, make the following replacement at classical crossings, 
	\begin{center}
		\includegraphics[scale=0.65]{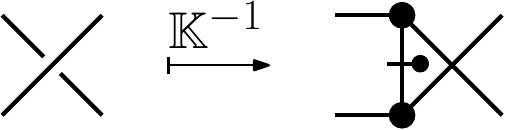}
	\end{center}
Replace all virtual crossings of the diagram with intersections between unmatched edges.

A morphism in \( D ( \mathbb{L} ) \) is a finite sequence of virtual Reidemeister moves. The functor \( \Kinv \) acts on such a morphism by applying the above replacement at every stage.
\end{definition}

Notice that the matching given by the replacement of \Cref{Def:kmap2} is a perfect matching: the unmatched edges correspond to the arcs of the virtual link diagram, while the matched edges are precisely those introduced by the replacement. Therefore if an edge corresponds to an arc of the argument virtual link diagram it is not a matched edge. It follows that the induced matching is perfect.

The reader is encouraged to apply \( \Kinv \) to a classical Reidemeister move: they will obtain one of the graphene moves \( G1 \), \( G2 \), or \( G3 \). This observation is made concrete in \Cref{Prop:welldefined2}.

\begin{lemma}\label{Lem:othercrossing}
The functor \( \mathbb{K}^{-1} \) induces the following:
\begin{center}
\includegraphics[scale=0.65]{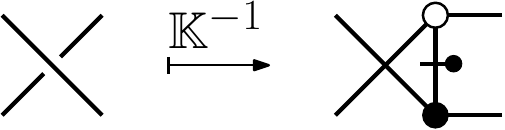}
\end{center}
\end{lemma}
\Cref{Lem:othercrossing} follows from \Cref{Lem:rotate}.

The functor \( \mathbb{K}^{-1} \) descends to a functor on virtual links. 
\begin{proposition}\label{Prop:welldefined2}
Let \( D_1 \) and \( D_2\) be virtual link diagrams. If \( D_1 \) is related to \( D_2 \) via virtual Reidemeister moves, then \( \mathbb{K}^{-1} ( D_1 ) \) is related to \( \mathbb{K}^{-1} ( D_2 ) \) via graphene moves.
\end{proposition}

\begin{proof}
Let \( D_1 \) and \( D_2 \) be virtual link diagrams. We are required to show that if \( D_1 \) is related to \( D_2 \) via a finite sequence of virtual Reidemeister moves and planar isotopies, then \( \mathbb{K}^{-1} ( D_1 ) \) is related to \( \mathbb{K}^{-1} ( D_2 ) \) via a finite sequence of graphene moves and planar isotopies. Again, it suffices to verify this for sequences of length one.

The moves \( V1 \), \( V2 \), and \( V3 \) of \Cref{Fig:vrms} are sent to the ribbon moves \( M1 \), \( M2 \), and \( M3 \). The move \( VM \) is dealt with by reversing the section of the proof of \Cref{Prop:welldefined1} dealing with \( M4 \).

The classical Reidemeister moves $R1$ to $R3$ depicted in \Cref{Fig:crms} are sent to the graphene moves \( G1 \) to \( G3 \) respectively. We verify this for the \( R1 \) move, leaving the other cases to the reader. We have 
\begin{center}
\includegraphics[scale=0.65]{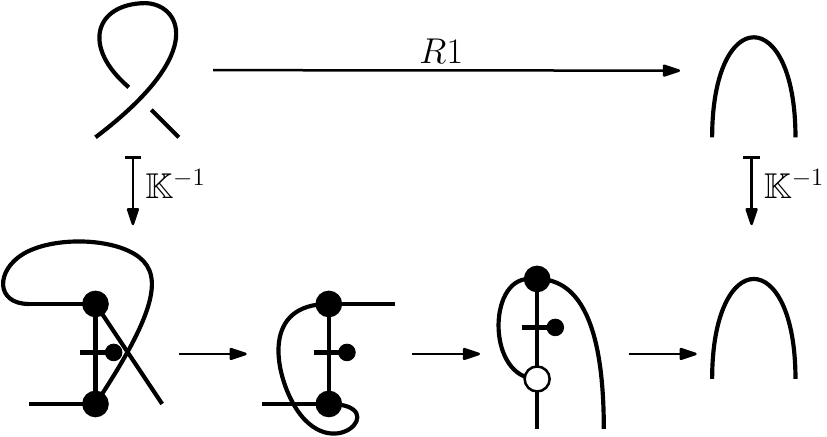}
\end{center}
via \Cref{Lem:detour}, \( M 5 \), and \( G1 \).
\end{proof}
As with \( \K \), we abuse notation to denote by \( \mathbb{K}^{-1} \) the functor from \( \mathbb{L} \) to \( \mathbb{G} \) induced by \Cref{Def:kmap2}.

We now verify that the functors \( \mathbb{K} \) and \( \mathbb{K}^{-1} \) are inverse to one another.
\begin{theorem}\label{Prop:inverse}
For all graphenes \( \mathcal{G} \) and virtual links \( L \) we have
\begin{equation*}
	\begin{aligned}
		\mathbb{K}^{-1} \left( \mathbb{K} \left( \mathcal{G} \right) \right) &= \mathcal{G} \\
		\mathbb{K} \left( \mathbb{K}^{-1} \left( L \right) \right) &= L
	\end{aligned}
\end{equation*}
\end{theorem}

\begin{proof}
We prove the case where the perfect matching edge is positive and leave the negative perfect matching edge case to the reader. We have 
\begin{center}
\includegraphics[scale=0.65]{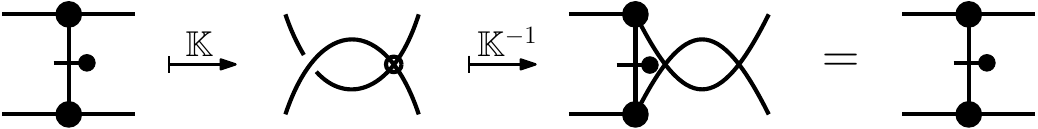}
\end{center}
and
\begin{center}
\includegraphics[scale=0.65]{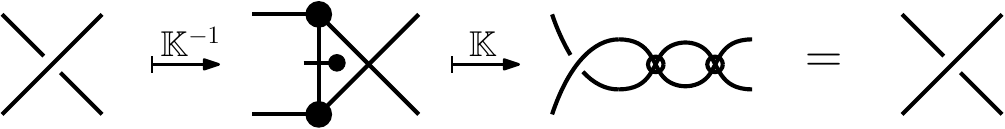}
\end{center}
\end{proof}
It follows that the categories \( \mathbb{G} \) and \( \mathbb{L} \) are isomorphic.

\begin{corollary}\label{Cor:bijection}
	The functor \( \K \) yields an isomorphism between the categories \( \mathbb{G} \) and \( \mathbb{L} \).
\end{corollary}

\begin{example}\label{Ex:inverse}
The image of the trefoil under \( \mathbb{K}^{-1} \) is the following graphene:
\begin{center}
\includegraphics[scale=0.5]{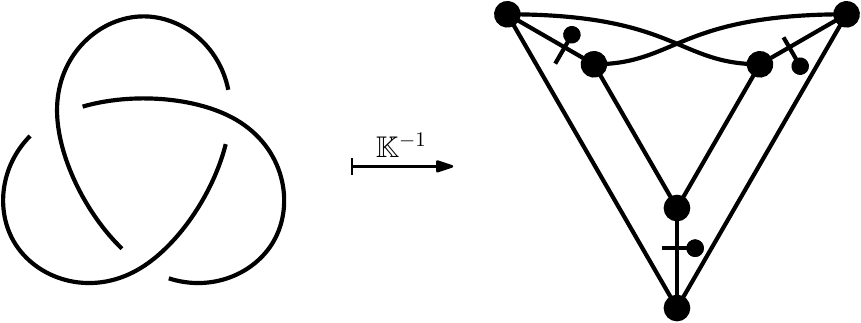}
\end{center}
Notice that the underlying graph of the graphene is bipartite and a $K_{3,3}$.
\end{example}

A perfect matching is \emph{even} if every cycle in the complement of the matching set is of even length. Notice that given a classical link \( L \), the graphene \( \mathbb{K}^{-1} ( L ) \) possesses an even perfect matching. This is depicted in \Cref{Ex:inverse}. Even perfect matchings are equivalent to \(3\)-edge-colorings, so that we obtain the following.

\begin{corollary}
	Let \( \G \) be a graphene such that \( \K ( \G ) \) is a classical link. Then the underlying graph of \( \G \) possesses an even perfect matching and is 3-edge-colorable.
\end{corollary}

A virtual link is sometimes said to be {\em even} if each component of a link has an even number of classical crossings between itself and the other components.  The $\K$ functor takes an even virtual link to a graphene with an even perfect matching. This relationship is further explored in \Cref{Sec:embeddings}.

\begin{definition}[Ungraph]\label{Def:ungraph}
The graph with no vertices and \( n \) edges is known as the \emph{ungraph of \( n \) components}.
\end{definition}
We refer to a graphene which has the ungraph of \( n \) components as a matched diagram representative as \emph{the ungraphene of \( n \) components}. The ungraphene of one component is known simply as the ungraphene.

\begin{example}\label{Ex:franklin}
The Franklin graph has a perfect matching which yields the ungraphene:
\begin{center}
\includegraphics[scale=0.55]{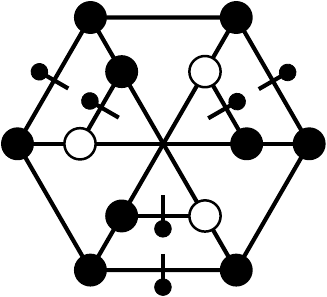}
\end{center}
This matched diagram can be reduced to the ungraph using three \( G 2 \) moves.
\end{example}

\subsection{\(Z\)-virtual links} \label{section:freevirtuallinks}
In \Cref{Def:dpm} we introduced the notion of a directed perfect matching: a decoration on the matched edges of a perfect matching. This decoration is necessary if \( \mathbb{K} \) is to be an isomorphism from the category of graphenes to that of virtual links. However, one may remove this requirement by considering a certain quotient of the category of virtual links.

\begin{definition}
	\label{Def:zmove}
	The following move on virtual link diagrams is known as the \emph{$Z$-move}
	\begin{center}
		\includegraphics{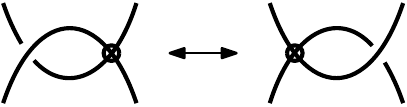}
	\end{center}
\end{definition}

A \emph{\(Z\)-virtual link} is an equivalence class of virtual link diagrams up to the virtual Reidemeister moves and the $Z$-move. We denote the category of \(Z\)-virtual links by \( \widetilde{\mathbb{L}} \).

Consider a ribbon diagram with a perfect matching without a direction. Consider the moves obtained from those given in \Cref{Def:graphenemoves} by removing the direction of the matched edges. These moves, together with those given in \Cref{Def:ribbonmoves}, form the \emph{\(Z\)-graphene moves}. A \emph{\(Z\)-graphene} is an equivalence class of ribbon diagram with undirected perfect matchings, up to the \(Z\)-graphene moves. Denote the category of \(Z\)-graphenes by \( \widetilde{\mathbb{G}} \).

As stated above, removing the direction of a perfect matching causes \( \mathbb{K} \) to fail to be an isomorphism. Adding the $Z$-move recovers this property. Notice that replacements given in \Cref{Def:kmap1} define a functor on \(Z\)-graphenes, simply by forgetting the direction; we shall denote this functor by \( \widetilde{\mathbb{K}} \).

\begin{theorem}\label{Thm:free}
	The functor \( \widetilde{\mathbb{K}} : \widetilde{\mathbb{G}} \rightarrow \widetilde{\mathbb{L}} \) from \(Z\)-graphenes to \(Z\)-virtual links is an isomorphism.
\end{theorem}

\Cref{Thm:free} may be proved using techniques essentially identical to those of \Cref{Subsec:kmap}.

Passing to \(Z\)-graphenes and \(Z\)-virtual links has the advantage of removing the dependence of the construction on the direction on the perfect matchings. This is desirable as it moves the objects of study closer to trivalent graphs themselves. However, in \Cref{Sec:invariants}, invariants of virtual links are pulled back via \( \mathbb{K}^{-1} \) to yield invariants of graphenes. To replicate this in the case of \(Z\)-graphenes one must use invariants of \(Z\)-virtual links. While such invariants are known, they are less plentiful supply than virtual link invariants.

\section{Invariants of graphs from graphenes and virtual links}\label{Sec:invariants}
\Cref{Prop:welldefined1,Prop:welldefined2} and \Cref{Prop:inverse} establish that \( \K \) is an isomorphism between graphenes and virtual links. It follows that any invariant of virtual links may be pulled back to yield an invariant of graphenes and vice versa.

\begin{theorem}\label{Thm:invariants}
Let \( f:\BL \ra S \) be an invariant of virtual links with target category $S$. The functor \( f \circ \K \) is an invariant of graphenes. Likewise, let $g: \BG \ra S' $ be any invariant of graphenes with target category $S'$. The functor $g\circ \K^{-1}$ is an invariant of virtual links.
\end{theorem}
Modern virtual knot theory is a large subject, and a wealth of invariants of many different flavours have been defined (for an introduction see \cite{Manturov2012,Kauffman2012Intro}). These include
\begin{itemize}
	\item Alexander-like invariants \cite{Kauffman,Boden15,Sawollek,Silver01,Silver03}
	\item Jones-like invariants \cite{Dye2005,Miyazawa08,Boden2020,Kauffman2013}
	\item Homological invariants \cite{Manturov2006,Dye2014,Rushworth2017,Manturov15,Dye11, Baldridge2020}
	\item Parity-based invariants \cite{Manturov2010,Ilyutko2013,Manturov2012,Manturov13}
\end{itemize}
Our discussion here is limited to just two invariants. We begin with the fundamental group of a graphene: this construction is notable as it yields a topological invariant of an abstract graph. We also describe the binary bracket polynomial of a virtual link, and its relationship to graph coloring.

\subsection{The fundamental group}
Let \( L \) be a link in \( S^3 \). The fundamental group of the link complement \( \pi_1 \left( S^3 \setminus L \right) \) is one of the most well-studied and important invariants of classical links. Indeed, it provides a complete invariant for knots \cite{GL1989}. The group \( \pi_1 \left( S^3 \setminus L \right) \) is known as the \emph{link group}.

The construction of the link group extends to virtual links. While it no longer yields a complete invariant of virtual knots, it is nevertheless an interesting and useful tool. Using \Cref{Thm:invariants}, this yields a group-valued invariant of graphenes, that we outline below.

Let \( L \) be a virtual link. From a diagram of \( L \) one may compute a group, with generators for classical crossings and relations given by the way in which these crossings are connected. This group is an invariant of \( L \); for full details see \cite[Section 4]{Kauffman1998}. Denote by \( \pi_1 ( L ) \) the group arrived at in this manner, known as the \emph{fundamental group of \( L \)}. Note that this group is defined combinatorially; as we describe below, however, it is the fundamental group of a certain topological space.

We may now apply \Cref{Thm:invariants} to make the following definition.
\begin{definition}[Fundamental group of a graphene]
Let \( \mathcal{G} \) be a graphene. Define \emph{the fundamental group of \( \mathcal{G} \)}, denoted \( \pi_1 ( \mathcal{G} ) \), as
\begin{equation*}
	\pi_1 ( \mathcal{G} ) \coloneqq \pi_1 ( \mathbb{K} ( \mathcal{G} ) ).
\end{equation*}
\end{definition}

The group \( \pi_1 ( \mathcal{G} ) \) has the following advantageous property. Recall that \( \mathcal{G} \) is an equivalence class of graphs. While these graphs are given an oriented ribbon structure and a directed perfect matching, at no point were they embedded into an ambient space. In other words, our objects of study are abstract graphs. As such, the group \( \pi_1 ( \mathcal{G} ) \) contains intrinsic information regarding the abstract graphs, rather than relative information regarding an embedding.

On the other hand, the group \( \pi_1 ( \mathcal{G} ) \) is the \emph{bona fide} fundamental group of a topological space. To see this, recall that virtual links may be represented as equivalence classes of links in thickened surfaces (as described in \Cref{Subsec:vkt}). Suppose that \( L \) is a virtual link that has as a representative a link in \( \Sigma_g \times I \). Denote this representative by \( D \hookrightarrow \Sigma_g \times I \). After setting 
\begin{equation*}
	M = \faktor{\Sigma_g \times I}{\Sigma_g \times \lbrace 0 \rbrace},
\end{equation*}
it is demonstrated in \cite[Proposition 5.1]{Kamada2000} that the (combinatorially defined) group \( \pi_1 ( L ) \) is isomorphic to the group \( \pi_1 ( M \setminus D )\), the honest fundamental group of a topological space.

In summary, while the graphene \( \mathcal{G} \) is an abstract object without a chosen embedding, the group \( \pi_1 ( \mathcal{G} ) \) is the fundamental group of a topological space associated to \( \mathcal{G} \).

\begin{question}
	What graph-theoretic information is contained in the group \( \pi_1 ( \mathcal{G} ) \)?
\end{question}

Related to the fundamental group of a virtual link is the notion of the \emph{fundamental quandle} \cite{Joyce82}. A quandle is an algebraic structure that is particularly well-suited for generalizing the link group. The notion of a quandle has itself been generalized in many directions within virtual knot theory, all of which apply, via \Cref{Thm:invariants}, to graphenes.

\subsection{The Binary Bracket Polynomial}
In this section we recall the definition of a variant of the bracket polynomial \cite{Kauffman1987a}, called the {\em binary bracket polynomial} \cite{Kauffman2004b}, and describe its relation to graph colorings.

Let $D$ be an unoriented (virtual) link diagram. Define an {\em unlabeled state}, $S$, of $D$  to
be a choice of smoothing for each  classical crossing of $K.$ There are two choices for smoothing a given  crossing, and
thus there are $2^{N}$ unlabelled states of a diagram with $N$ classical crossings.  A {\em labeled state} is a state $S$ such
that the labels $0$ or $1$ have been assigned to each copy of \( S^1 \) in the state; an example is given in the bottom right of \Cref{Figure 4}.

\begin{figure}
	\begin{center}
		\begin{tabular}{c}
			\includegraphics[width=6cm]{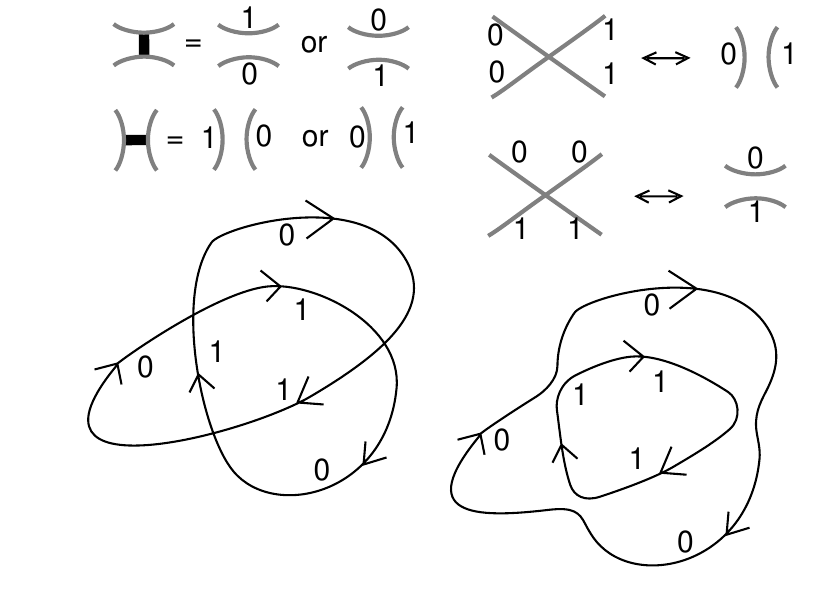}
		\end{tabular}
		\caption{Bracket smoothings}
		\label{Figure 4}
	\end{center}
\end{figure}

In a state we designate each smoothing with $A$ or $A^{-1}$ according to the left-right convention shown in the first line of \Cref{Eq:bin1}. This designation is called a {\em vertex weight} of the state. (We require of a labelled state that the two labels that occur at a smoothing of a crossing are distinct.) This is indicated by a bold line between the arcs of the smoothing as illustrated in \Cref{Figure 4}. Labelled states satisfying this condition at the site of every smoothing will be called {\em properly labeled states.} If $S$ is a properly labeled state, we let $\{ D|S \}$ denote the product of its vertex weights, and we define the two-color bracket polynomial by the equation:   
$$\{ D \} \, = \sum_{S} \{ D|S \}$$
where $S$ runs through the set of properly labelled states of $D$.

It follows from this definition that $\{ D \}$ satisfies the equations
\begin{equation}
	\begin{aligned}
		&\includegraphics[width=6cm]{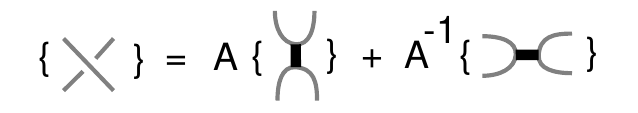} \label{Eq:bin1}\\
		&\{ D \sqcup  O \} \, = 2 \{ D \},\\
		&\{ O \} \, =2.\\
	\end{aligned}
\end{equation}

In the first equation, we indicate that the colors at the smoothing are different by the dark band placed between the arcs of the smoothing.

Define the {\em binary bracket polynomial} as $\{ D \} \, = \, \{ D \}(A)$. It associates to each unoriented virtual link diagram $D$ a Laurent polynomial in the variable $A$.

Let \( D \) be an oriented diagram of the virtual link \( L \). The binary bracket can be normalized to an invariant of oriented virtual links via the definition $$Inv_{L}(A) = A^{-w(D)} \{D \}(A).$$ The resulting invariant has a number of interesting properties, as described in \cite[Theorem 2]{Kauffman2004b}. These properties include detection of non-classicality, detection of chirality, and the counting of particular colorings of the argument link.

\begin{remark}
A virtual link diagram is colorable in two colors if and only if the diagram is {\it even}, meaning that the number of common crossings between any two components is even.

As is discussed in \Cref{Sec:51}, this condition is equivalent to the three-colorablity of the corresponding graphene with all perfect matching edges receiving the same color. Note also that when the virtual link diagram is not even, then there are no two colorings and the binary bracket vanishes. When  $A=1$, the binary bracket counts the number of $2$-colorings of the link diagram, and hence, when the link is even, is equal to $2^{N}$ where $N$ is the number of link components. These results are in parallel with the corresponding results for the Jones polynomial, discussed in \Cref{Sec:51}.
\end{remark}

\begin{figure}
	\begin{center}
		\begin{tabular}{c}
			\includegraphics[width=3cm]{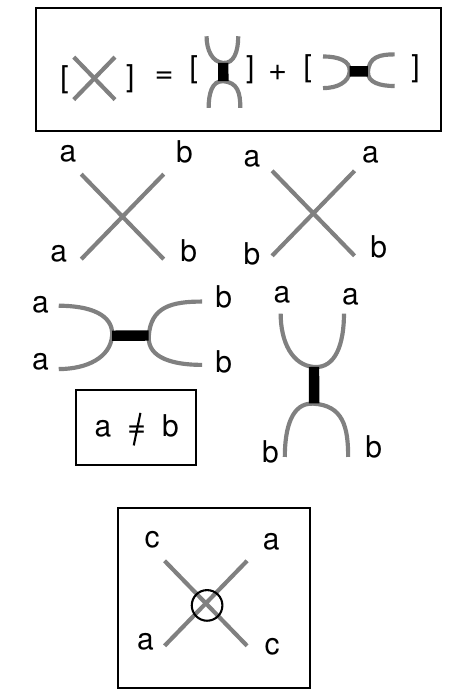}
		\end{tabular}
		\caption{Coloring rules at flat and virtual crossings.}
		\label{Figure 15}
	\end{center}
\end{figure}

\begin{figure}
	\begin{center}
		\begin{tabular}{c}
			\includegraphics[width=4cm]{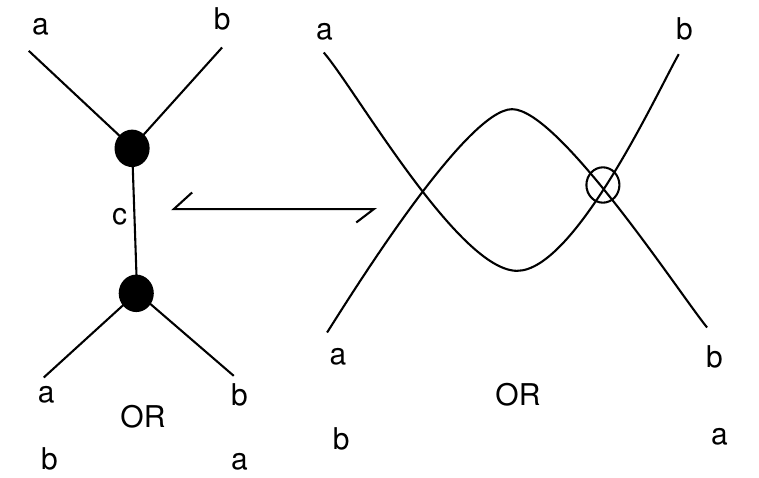}
		\end{tabular}
		\caption{Translation between cubic graphs and shadow virtual diagrams.}
		\label{Figure 21}
	\end{center}
\end{figure}

\begin{figure}
	\begin{center}
		\begin{tabular}{c}
			\includegraphics[width=4cm]{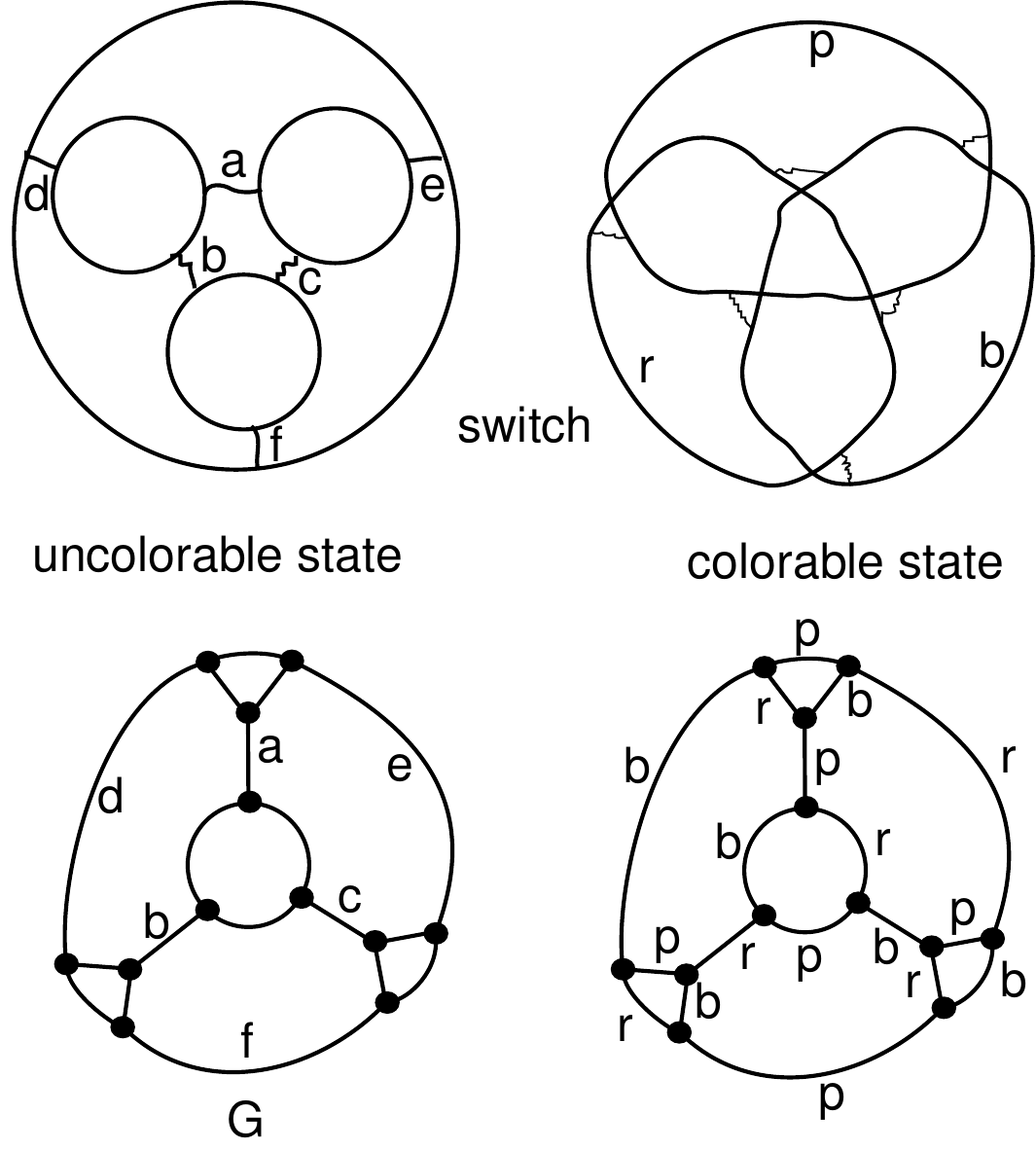}
		\end{tabular}
		\caption{Coloring a planar state and its graph.}
		\label{three}
	\end{center}
\end{figure}

\subsubsection{Colorings and generalizations}
If we allow the binary bracket to use three colors it is no longer an invariant of virtual links, but it is directly related to the coloring problem for trivalent graphs. The coloring expansion in \Cref{Figure 15} corresponds to an expansion for trivalent graphs, as depicted in \Cref{Figure 21}. This tautological expansion applies to any trivalent graph. It can be expanded by choosing a perfect matching but does not depend upon this choice.
The formula \cite{KP,StateCalculus,Map} is given below. 
$$\{ \YGlyph  \} = \{ \VGlyph \} + \{  \CGlyph  \}$$

The meaning of this equation is as follows: $\{ G \}$ denotes the number of proper 3-colorings of the edges of a trivalent graph $G.$ The graph $G$ is not assumed to be planar nor represented in the plane. The left-hand term of the formula represents a graph $G$ and one distinguished edge of $G.$ The two right hand terms of the formula represent the graph $G$ with the specific edge replaced either by a parallel connection of edges as illustrated, or by a cross-over connection of edges. In each case the resulting edges are to be colored with different colors. Thus the two graphs on the right-hand side of the formula have acquired a local edge 3-coloring. Repeated use of the formula shall result in a global edge 3-coloring, if one exists. The 3-colorings that appear in this manner are related to the generators of the homology theory described in \Cref{Sec:embeddings}.

In \Cref{three} we illustrate a planar trivalent graph and two possible results of the formula given above. One result consists in a collection of disjoint circles in the plane and is not edge 3-colorable. The other state is colorable, showing that the original trivalent graph is three-colorable.

\section{Homology theories}\label{Sec:homologytheories}
In this section we discuss homology theories of virtual links, all of which yield homology theories of graphenes via \(\K \). We also discuss the relationship between such theories and a homology theory of trivalent graphs with perfect matchings due to the first author.

The introduction of the Jones polynomial is one of the most important developments in low dimensional topology. A polynomial invariant of classical links strong enough to resolve the Tait conjectures, it has a central place in modern knot theory. Khovanov lifted the Jones polynomial to a group-valued invariant, a process known as categorification. Specifically, he defined a chain complex associated to a link diagram, the homology of which is a link invariant. This homology is a finitely generated bigraded Abelian group, the graded Euler characteristic of which recovers the Jones polynomial. Now known as Khovanov homology, this construction has also proved to be a powerful tool in the study of links and low dimensional manifolds.

Using \( \K \) it is possible to pull back such homology theories to graphenes. It is reasonable to suspect that invariants defined in this way will contain interesting graph-theoretic information. An example of this phenomenon is described in \Cref{Sec:embeddings}.

\subsection{Jones polynomial and Khovanov homology of graphenes}\label{Sec:51}
In this section we describe the Jones polynomial and Khovanov homology of a graphene.  The Jones polynomial and Khovanov homology are invariants of virtual links that depend on orientation. The corresponding notion of orientation for graphenes is an orientation of the cycles in $G\setminus M$ of the graphene $\graphene$.  (The functor $\K$ takes the cycles in $G\setminus M$ of a matched diagram $\Gamma$  of $\graphene$ to the components of the virtual link diagram $\K(\Gamma)$.)  Thus, a graphene $\graphene$ can be given a {\em cycle-orientation} by choosing an orientation for each cycle in $G\setminus M$. There are $2^\ell$ such orientations where $\ell$ is the number of cycles in $G\setminus M$.  A cycle-orientation of $\graphene$ induces an orientation of the corresponding virtual link $\K(\graphene)$, and this orientation can be used to define the {\em graphene Jones polynomial} $J_\graphene := J_{\K(\graphene)}$  and the {\em graphene Khovanov homology} $Kh(\graphene) : = Kh(\K(\graphene))$.  A consequence of \Cref{Thm:invariants} is the following.

\begin{theorem}
The graphene Jones polynomial and graphene Khovanov homology are invariants of graphenes with cycle-orientations.
\end{theorem}
 
There is something interesting that can be done with a ribbon graph that cannot be done with a virtual link diagram: we can consider all of the different perfect matchings of the ribbon graph.  For a given ribbon graph this produces a set of matched diagrams as follows. Let $\Gamma$ be a ribbon diagram  and let $\{M_1,\dots,M_p\}$ be the set of all perfect matchings of the underlying graph of $\Gamma$. Arbitrarily choose a direction for each perfect matching edge in each perfect matching set (we can avoid these choices by passing to $Z$-graphenes, as is done in \Cref{Sec:52}).  The corresponding matched diagrams are then $\{\Gamma_1,\dots,\Gamma_p\}$. 

\begin{question}
Do invariants of the matched diagrams $\{\Gamma_1,\dots,\Gamma_p\}$ contain combinatorial information about the ribbon graph or underlying graph?
\end{question}

This question has an affirmative answer. It can be shown that the graphene Jones polynomial of a matched diagram evaluated at 1 reports the number of $2$-factors of the underlying graph of the graphene that spans its perfect matching edges (cf. \cite{BaldridgeLowranceMcCarty2018}).  As described in \Cref{Sec:motivation}, this is equal to the number of 3-edge-colorings of the underlying graph that have the same fixed color on the perfect matching edges of a perfect matching (cf. \cite{Baldridge2018}). Summing up over all colorings of all perfect matchings sets gives the total number of 3-edge-colorings of the graph.  To compute this sum, however, a cycle-orientation $\sigma_i$ must be chosen for each matched diagram $\Gamma_i$.  Then just as in the case of the planar Tait polynomial in \cite{Baldridge2018} we obtain:

$$\sum_{i=1}^p J_{\Gamma_i, \sigma_i}(1) = \#\{\mbox{3-edge-colorings of } \Gamma\}.$$

One of the main motivations of  \cite{Baldridge2020} was to remove the need to choose a cycle-orientation $\sigma_i$ for each $\Gamma_i$.  In that paper, we described the {\em unoriented Jones polynomial} $\tilde{J}_L$, which is an invariant of a virtual link $L$ that does not require choosing an orientation on the components of the link.  By  \Cref{Thm:invariants}, there is a graphene version of this polynomial for matched diagrams and it also reports the number of $2$-factors that spans the perfect matching when evaluated at 1.  In \cite{Baldridge2} and other future papers, we shall describe how to sum these polynomials together to get a Tait polynomial as in \cite{Baldridge2018} that is an invariant of the trivalent ribbon graph.  

Another  useful result that came out of \cite{Baldridge2020} is a vastly simplified definition of Khovanov homology for virtual links.  We now discuss this homology next and how it compares to the homology defined by the first author.

\subsection{Khovanov homology versus the homology theory defined by Baldridge}\label{Sec:52}
One of the main motivations of this paper was to explore the relationship between the Khovanov homology of virtual links \cite{Khovanov1999, Manturov2006,Dye2014, Baldridge2020} and the homology theory defined by the first author on planar graphs \cite{Baldridge2018}.  The functor $\K$ allows us to make such a comparison these homology theories. In this section we show that, while the two homologies are isomorphic on suitable matched diagrams, the homology theories are different in that the equivalence classes of each theory are inequivalent: one cannot use the graphene moves to perform the flip moves of planar graphs. 

First, we show how the two theories can be brought into alignment with each other by regarding a planar matched diagram of a planar free graphene as a perfect matching drawing of a planar graph \cite{Baldridge2018}.   Since Khovanov homology is invariant under the $Z$-move of \Cref{Def:zmove}, we work in the categories of \(Z\)-virtual links and \(Z\)-graphenes.  A \(Z\)-graphene $\graphene$ can be represented by a matched diagram $\Gamma$ where the perfect matching edges are not given a direction.  By using the $M5$ and $G5$ moves, we can assume that this matched diagram $\Gamma$ has no hollow vertices or negative perfect matching edges.  When the associated surface to $\Gamma$ is $S^2$, then $\Gamma$ is a planar embedding of the graph $G$ with a perfect matching $M$.  Henceforth assume that $\Gamma$ is such a matched diagram.  We refer to this type of matched diagram as a  {\em perfect matching drawing}, as in \cite{Baldridge2018}.  A perfect matching drawing can regarded as a plane drawing of a graph.

There is a natural cycle-orientation of a perfect matching drawing. Orient the cycles of $G\setminus M$ in $\Gamma$ as follows.  The cycles are a set of disjoint circles embedded in the plane;  orient a circle counter-clockwise if it is separated from  infinity by an even number of circles, otherwise orient the circle clockwise.   This cycle-orientation of $\Gamma$ orients the virtual link diagram $\K(\Gamma)$ so that each classical crossing in $\K(\Gamma)$ is a positive crossing.  

We are now ready to compare the Khovanov homology $Kh(\Gamma)$ of the perfect matching drawing $\Gamma$ with the  homology $H(\Gamma)$ defined by Baldridge in \cite{Baldridge2018}. In \Cref{Sec:motivation} the $2$-factor bracket of $\Gamma$ was shown to be equal to the Kauffman bracket of its associated virtual link diagram $\K(\Gamma)$, i.e., $$\langle \Gamma \rangle_2 = \langle \K(\Gamma) \rangle.$$ Similarly, $Kh(\Gamma)$ and $H(\Gamma)$ can be calculated directly from a matched diagram $\Gamma$ without first applying \( \K \).  This calculation is described in \cite{Baldridge2018}: the hypercube of states of $\Gamma$ is built exactly as in Khovanov theory, except that the $0$-smoothing of a perfect matching edge of $\Gamma$ is \raisebox{-0.33\height}{\includegraphics[scale=0.20]{A-smoothing.pdf}} and the $1$-smoothing is \raisebox{-0.33\height}{\includegraphics[scale=0.20]{X-smoothing.pdf}}.  (These are the smoothings used in the $2$-factor polynomial in \Cref{eq:bracket-2-factor}.)  Each state in this hypercube has exactly the same number of circles as the corresponding Khovanov homology state of $\K(\Gamma)$, and the $\Z_2$-vector spaces and maps to and from them also remain the same.  The homology defined from the chain groups defined in this way is the  homology $H(\Gamma)$.  It is ismorphic to the Khovanov homology of $\Gamma$ with $\Z_2$-coefficients, but with a shift in the second grading.

\begin{theorem}[{\cite[Section 5]{Baldridge2018}}] Let $\Gamma$ be a perfect matching drawing, with the natural cycle-orientation described above. Then
$$Kh^{i,j}(\Gamma) = H^{i,j+n}(\Gamma),$$
where \( Kh \) denotes Khovanov homology with $\Z_2$-coefficients, and $i$ the homological grading, $j$ the $q$-grading, and $n$ is the number of edges in the perfect matching of $\Gamma$.
\label{Theorem:diagrammatic-iso}
\end{theorem}

The above theorem only holds for perfect matching drawings. Indeed, the graphene moves $G1$ and $G2$ change the underlying graph, so that performing either of these moves on $\Gamma$ breaks the hypothesis of \Cref{Theorem:diagrammatic-iso}, in general.

The remaining question is whether the $0$-, $1$-, and $2$-flip moves (cf. \cite{Baldridge2018}) on perfect matching drawings can be accomplished through a sequence of graphene moves.  If true, then planar trivalent graphs with perfect matchings would constitute a restriction to a subcategory of virtual links and a set of  Reidemeister moves that preserve planarity.  By \Cref{Lem:orientationflip}, a $0$-flip move that takes a component of $\Gamma$ to its mirror can be done with graphene moves.  Using similar thinking and calculations, one can attempt to perform a $2$-flip (or a $1$-flip) using graphene moves---the flipping region can be flipped over except for the introduction of a ``twist'' on the top and bottom edges, as depicted in the central diagram of \Cref{Fig:2-flip-move-example}.

\begin{figure}
\includegraphics[scale=0.7]{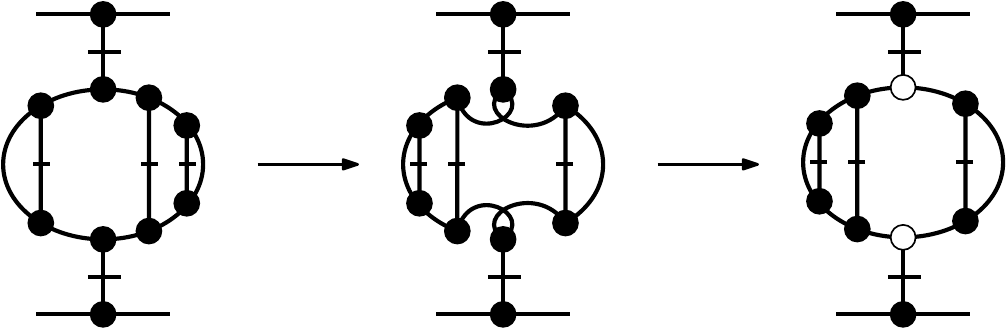}
\caption{Attempting a 2-flip using only graphene moves. }
\label{Fig:2-flip-move-example}
\end{figure}

The rightmost diagram of \Cref{Fig:2-flip-move-example} shows that one can do a $2$-flip using graphene moves up to introducing two hollow vertices on the top and bottom edges.  Removing these hollow vertices is the problem.  A positive matching edge with a hollow  and solid vertex is equivalent to a negative perfect matching edge with two solid vertices by $G5$.  We would need a move that takes a positive perfect matching edge to a negative perfect matching edge that preserves the graphene and graphene Khovanov homology.  However, comparing \( \K \) in \Cref{Def:kmap1} on positive and negative edges, this move on virtual links would be equivalent to the move depicted in \Cref{Fig:virtual-switch}.

\begin{figure}
\includegraphics[scale=0.7]{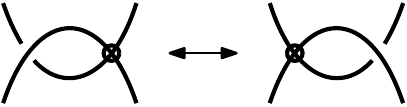}
\caption{The virtual switch move does not preserve a virtual link in general. In terms of graphenes, it is equivalent to switching the perfect matching edge from positive to negative.}
\label{Fig:virtual-switch}
\end{figure}

This move is  {\em not} a $Z$-move, but what the second author called a {\em virtual switch} \cite{Kauffman2001}.  A virtual switch cannot be performed using Reidemeister moves.  In fact, it is well known that the Jones polynomial and Khovanov homologies are not invariant under such a move. Thus, we arrive at the fact that performing a 2-flip move on a perfect matching drawing cannot be accomplished through Reidemeister moves: the  homology theory of the first author cannot be couched in the theory of virtual links despite the strong connections evidenced by \Cref{Theorem:diagrammatic-iso}.

\begin{theorem}
Let $\Gamma_1$ and $\Gamma_2$ be two perfect matching diagrams of a planar trivalent graph with perfect matching $(G,M)$.  If $\Gamma_1$ and $\Gamma_2$ are related by a 2-flip move, then  $\Gamma_1$ and $\Gamma_2$ are  inequivalent as graphenes using the graphene moves.
\end{theorem}

Using \( \K \), a  $2$-flip move on the perfect matching drawing $\Gamma$ {\em does} correspond to a special type of ``mutation'' on the virtual link $\K(\Gamma)$, as depicted in \Cref{Fig:2-flip-mutation}.

\begin{figure}
\includegraphics[scale=1.2]{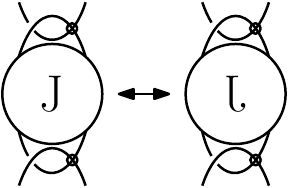}
\caption{A $2$-flip on perfect matching diagrams is equivalent to a special type of mutation on a virtual link.}
\label{Fig:2-flip-mutation}
\end{figure}

Khovanov homology with $\Z_2$-coefficients is invariant under mutations for classical links  \cite{Bloom2010, Wehrli2010} (see also \cite{Lambert-Cole2019}). However, there is no reason to expect invariance in the virtual category. A mutation of a classical link removes a $3$-ball from the $3$-sphere whose boundary intersects the link in four points and glues the ball back using an involution of the boundary.  As described in \Cref{sec:virt-links-in-thick-surfaces}, virtual links are links in thickened surfaces (up to stabilization).  Since a virtual tangle is not necessarily contained in a $3$-ball, there is no obvious analogous operation to mutation.  However, if the tangle and mutation looks like the picture in \Cref{Fig:2-flip-mutation}, then the $\Z_2$-coefficient Khovanov homology is invariant under such a mutation. Thus, the main theorem of \cite{Baldridge2018} can be seen as the first theorem in the literature of a mutation-like operation on virtual tangles that preserves the $\Z_2$-coefficient Khovanov homology of a virtual link.

\begin{theorem}[Theorem 4.3 of \cite{Baldridge2018}]
Let $L$ be an oriented virtual link, $D$ a virtual link diagram of $L$, and $J$ a virtual tangle of $D$ as pictured on the left of \Cref{Fig:2-flip-mutation}.  Let $D'$ be the corresponding link diagram as pictured on the right of \Cref{Fig:2-flip-mutation}, but with the opposite orientation on every component of $J$.  Then $$Kh(D) \cong Kh(D').$$ A similar statement holds for the image of a $1$-flip move under $\K$.
\end{theorem}
  
In this way we see the main difference between invariants of virtual links and invariants of planar $3$-regular (trivalent) graphs:  link invariants are invariants of the virtual Reidemeister moves and planar graphs invariants are invariants of special mutations of their corresponding virtual links. The $2$-factor polynomial and homology theory of the first author can be generalized to $k$-regular planar graphs for $k>3$, which have no counterpart in virtual link theory. Hence it is good to think of Khovanov homology and the homology theory defined by Baldridge as separate theories.  But they come together nicely in the case $k=3$. We shall discuss this relationship further in \cite{Baldridge2}.

\section{Strong embeddings and doubled Lee homology}\label{Sec:embeddings}
In this section we describe the relationship between strong embeddings of graphs and a homology theory of virtual links. We demonstrate that this relationship may be used to associate a pair of integers to a strong embedding. When working with virtual links, these integers contain deep geometric information. It is tantalizing to ask what graph-theoretic information they contain.

In \Cref{Sec:strong+cycles} we describe the association between strong embeddings and certain cycles of a graph. In \Cref{Sec:dkh} we outline the homology theory of virtual links known as doubled Lee homology, before exhibiting its connection with strong embeddings in \Cref{Sec:strong+dkh}.

\subsection{Bicolored multicycles yield strong embeddings}\label{Sec:strong+cycles}
The complement of a perfect matching in a graph is a union of cycles. We demonstrate that a particular coloring of such a union of cycles yields a strong embedding of the graph. We begin by recalling the definition of a strong embedding.

\begin{definition}
	An embedding of a graph \( G \) in a closed surface \( \Sigma \) is a \emph{strong embedding} if the complement \( \Sigma \setminus G \) is a disjoint union of discs. That is, the boundary of each face of a strong embedding is a simple cycle.
\end{definition}
A strong embedding is also known as a \emph{cellular embedding}.

Let \( G \) be a trivalent graph and \( M \) a perfect matching of it. The complement \( G \setminus M \) is a disjoint union of cycles of \( G \). A \emph{bicolored multicycle} is a coloring of each of the edges of \( G \setminus M \) by exactly one of two colors such that each vertex has one incoming edge of each color. A bicolored multicycle may be formed from a Tait coloring by declaring the edges of one color as the matching (see also \cite[Section 4.1]{Jaeger1985}). Let \( B ( G,M ) \) be the set of bicolored multicycles formed from \( G \setminus M \). It is clear that \( B ( G , M) \neq \emptyset \) if and only if \( M \) is even.

\begin{proposition}\label{Prop:bcm2c}
	Let \( G \) be a trivalent graph, \( M \) a perfect matching of it, and \( B ( G,M ) \) the associated set of bicolored multicycles. To every \( C \in B ( G,M ) \) there is an associated strong embedding of \( G \).
\end{proposition}

\begin{proof}
Given \( C \in B ( G,M ) \) we construct a strong embedding of \( G \) as follows. Arbitrarily direct the cycles making up \( C \) (that is, take a cycle-orientation). There are four possible configurations at matched edges
\begin{equation}\label{Eq:config1}
\begin{matrix}
\includegraphics[scale=0.7]{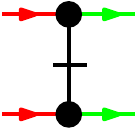}~ & \includegraphics[scale=0.7]{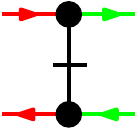}~ &\includegraphics[scale=0.7]{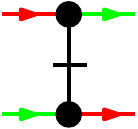}~ & \includegraphics[scale=0.7]{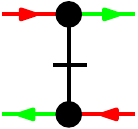}~ \\
\text{(i)} & \text{(ii)} & \text{(iii)} & \text{(iv)}
\end{matrix}
\end{equation}
Form a surface that deformation retracts onto \( G \) using the following components at matched edges, depending on the configuration
\begin{equation}\label{Eq:config2}
\begin{matrix}
\includegraphics[scale=0.7]{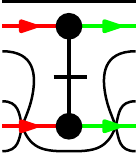}~ & \includegraphics[scale=0.7]{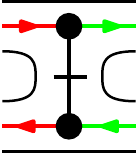}~ &\includegraphics[scale=0.7]{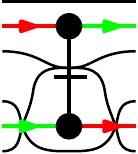}~ & \includegraphics[scale=0.7]{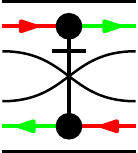}~ \\
\text{(i)} & \text{(ii)} & \text{(iii)} & \text{(iv)}
\end{matrix}
\end{equation}
and attaching bands with the edges as their cores. Denote the surface obtained by \( F_C \). An example is given in \Cref{Sec:example}.

Label the boundary components of \( F_C \) as follows, depending on the coloring and orientation of the edges
\begin{equation}\label{Eq:label}
\includegraphics[scale=1]{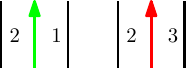}
\end{equation}
By the construction of the surface components in \Cref{Eq:config2} there exists a coherent labelling across \( F_C \), satisfying \Cref{Eq:label}. Notice that given such a coherent labelling the boundary components parallel to each edge possess distinct labels. It follows that by attaching discs to the boundary components of \( F_C \) we obtain a strong embedding of \( G \).
\end{proof}

\subsection{Doubled Lee homology}\label{Sec:dkh}
Khovanov homology is a powerful invariant of links in \( S^3 \) \cite{Khovanov1999}. Lee defined a perturbation of Khovanov homology, and used it to prove a conjecture posed by Bar-Natan regarding the homology of alternating knots \cite{Lee2005,BarNatan2002}. Rasmussen showed that the homology theory defined by Lee contains deep geometric information, contrary to initial expectations \cite{Rasmussen2010}.

There are two known extensions of Lee homology to virtual links; in this section we briefly sketch one of them. We shall treat the extension as a black-box, focussing on a combinatorial result concerning its rank as a group.

Let \( L \) be an oriented virtual link. As constructed in \cite[Definition 3.1]{Rushworth2017}, the \emph{doubled Lee homology} of \( L \), denoted \( \dkh ' ( L ) \), is a finitely generated bigraded Abelian group.

The group \( \dkh ' ( L ) \) is a perturbation of a theory known as \emph{doubled Khovanov homology}, and may be computed from a diagram of \( L \). Given a diagram, a formal chain complex of diagrams is produced, before being converted into algebra. This conversion yields a \emph{bona fide} chain complex, the homology of which is the doubled Khovanov homology of the link \( L \). To obtain \( \dkh ' ( L ) \), one passes to the \( E_{\infty} \)-page of a spectral sequence beginning at doubled Khovanov homology.

The rank of the group \( \dkh ' ( L ) \) depends only on a simple combinatorial property of \( L \).
\begin{definition}\label{Def:2cable}
	Let \( D \) be a virtual link diagram. We say that \( D \) is \emph{\(2\)-colorable} if its arcs may be colored exactly one of two colors such that at classical crossings we have
	\begin{center}
		\includegraphics{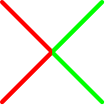}
	\end{center}
	up to rotation. The colors of arcs are unaffected by virtual crossings.
\end{definition}
An example of a \(2\)-colorable virtual link is given in \Cref{Fig:2c}.

\begin{figure}
	\includegraphics[scale=0.65]{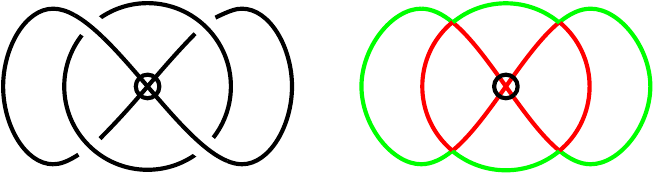}
	\caption{A virtual link diagram and a \(2\)-coloring of it.}
	\label{Fig:2c}
\end{figure}

The \(2\)-colorability of a virtual link may be checked directly from a diagram of it. Let \( L = K_1 \cup K_2 \cup \cdots \cup K_n \) be a virtual link of \( n \) components: \( L \) is \(2\)-colorable if and only if there is an even number of classical crossings between \( K_i \) and \( L \setminus K_i \), for all \( 1 < i < n \).

The rank of the doubled Lee homology of a virtual link is equal to the number of its \(2\)-colorings.
\begin{theorem}[{\cite[Theorem 3.5]{Rushworth2017}}]\label{Thm:leerank}
	Let \( L \) be an oriented virtual link and \( \dkh ' ( L ) \) its doubled Lee homology. Let \( \mathbb{O} \) denote the \( \mathbb{Q} \)-vector space generated by the \(2\)-colorings of \( L \). There is an isomorphism between \( \mathbb{O} \) and \( \dkh'(L) \).
\end{theorem}

There is a distinct extension of Lee homology to virtual links due Manturov and reformulated by Dye-Kaestner-Kauffman \cite{Manturov2006,Dye2014}. While the rank of their extension is equal to that of doubled Lee homology, the two theories are markedly different. For instance, they differ on the link depicted in \Cref{Fig:2c}.

\subsection{Bigradings of strong embeddings}\label{Sec:strong+dkh}
Combining the construction of \Cref{Sec:strong+cycles} and \Cref{Thm:leerank} allows the association of quantum-topological information to strong embeddings of graphs.

Recall that a graphene is a graph together with extra data. Let \( \mathcal{G} \) be a graphene with underlying graph \( G \). By an abuse of terminology we shall refer to a strong embedding of \( G \) as a strong embedding of \( \mathcal{G} \), and likewise a bicolored cycle of \( G \) as a bicolored cycle of \( \mathcal{G} \).

First, notice that bicolored multicycles behave well with respect to \( \K \).
\begin{proposition}\label{Prop:biject}
	Let \( \mathcal{G} \) be a graphene. The set of bicolored multicycles of \( \mathcal{G} \) and the set of \(2\)-colorings of \( \mathbb{K} ( \mathcal{G} ) \) are in bijection.
\end{proposition}
\begin{proof}
	The desired bijection is given by \( \mathbb{K} \) and \( \mathbb{K}^{-1} \), as follows
	\begin{equation*}
	\includegraphics[scale=0.75]{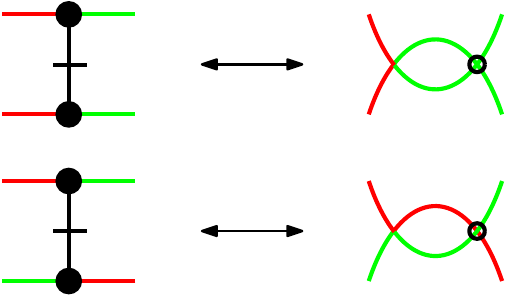}
	\end{equation*}
\end{proof}
As a direct corollary to \Cref{Thm:leerank} and \Cref{Prop:biject} we obtain the following.
\begin{proposition}\label{Prop:leecycle}
	Given a graphene \( \mathcal{G} \), the homology \( \dkh' ( \mathbb{K} ( \mathcal{G} ) ) \) is isomorphic to the \( \mathbb{Q} \)-vector space spanned by the bicolored multicycles of \( \mathcal{G} \).
\end{proposition}
We may therefore think of \( \dkh' ( \mathbb{K} ( \mathcal{G}) ) \) as being generated by the bicolored multicycles of \( \mathcal{G} \).

Further, recall that a bicolored multicycle defines a strong embedding. It follows that the non-vanishing of doubled Lee homology guarantees the existence of a strong embedding.
\begin{corollary}\label{Cor:strong}
	Let \( G \) be a graphene. If \( \dkh' ( \mathbb{K} ( \mathcal{G}) ) \) is not the trivial group then \( \mathcal{G} \) has a strong embedding.
\end{corollary}
This result allows us to guarantee that many graphs possess strong embeddings.
\begin{corollary}\label{Cor:strongprime}
	Let \( \G \) be a graphene such that \( \K ( \G) \) is a \(2\)-colorable virtual link. Then the underlying graph of \( \G \) has a strong embedding.
\end{corollary}
The family of \(2\)-colorable virtual links includes alternating links and checkerboard-colorable links. In the case of checkerboard-colorable virtual links it can be shown that the guaranteed strong embedding is orientable, in fact.

The doubled Lee homology of a virtual link is a bigraded group. Specifically, to each element of \( \dkh ' ( L ) \) there is an associated pair of integers \( (i,j) \). The integer \( i \) is known as the \emph{homological grading}; the homological grading support of \( \dkh ' ( L ) \) is easy to determine, and may be read off from a diagram of \( L \).

The second integer, \( j\), is known as the \emph{quantum grading}, so-called as it is related to the quantum representation theory of \( \mathfrak{sl}_2 \). In contrast to the homological grading, determining the quantum grading support of \( \dkh ' ( L ) \) requires computation of the homology groups.

As described above, doubled Lee homology is an extension of a theory known as Lee homology; the latter was originally defined for classical links, and is denoted \( Kh ' ( L ) \). The group \( Kh ' ( L ) \) is also bigraded, with a homological and a quantum grading.

Given a classical knot (a one-component link) the group \( Kh ' ( K ) \) is of rank \(2\), and the quantum grading of one generator determines that of the other.  Additionally, the homological grading of both generators is \( 0 \). Picking either generator, we see that the data of \( Kh ' ( K ) \) are equivalent to an integer: the quantum grading of the chosen generator. For full details see \cite[Section 3]{Rasmussen2010}.

This establishes that Lee homology associates an integer to a classical knot. Rasmussen demonstrated that, surprisingly, this integer contains deep geometric information about the knot. Specifically, given a knot \( K : S^1 \hookrightarrow S^3 \), one may consider smoothly embedded surfaces \( F \hookrightarrow B^4 \), such that \( \partial F = K \). Determining the minimum genus of such surfaces for a given knot is a difficult problem. Rasmussen showed that the integer assigned by Lee homology yields a very useful lower bound on this genus \cite[Theorem 1]{Rasmussen2010}. This integer is strong enough to provide a combinatorial proof of the Milnor conjecture, and is sensitive to smooth structures on \(4\)-manifolds \cite[page 13]{Rasmussen2005}.

Doubled Lee homology also contains interesting geometric information about virtual links \cite[Section 4]{Rushworth2017}. In light of \Cref{Sec:strong+cycles,Sec:dkh} is it interesting to ask if doubled Lee homology contains unexpected graph-theoretic information also.

One such line of enquiry is as follows. Let \( \mathcal{G} \) be a graphene with an even perfect matching, and \( C \) a bicolored multicycle of it. By \Cref{Prop:leecycle} there is a generator of \( \dkh' ( \mathbb{K} ( \mathcal{G} ) ) \) associated to \( C \). This generator has a bigrading, \( (i,j) \), and via the isomorphism of \Cref{Prop:leecycle} we may associate the pair \( (i,j) \) to \( C \) also. 

As demonstrated in \Cref{Sec:strong+cycles}, the bicolored multicycle \( C \) defines a strong embedding of \( \mathcal{G} \). It follows that doubled Lee homology associates pairs of integers to certain strong embeddings. It is natural to ask the following intriguing questions.
\begin{question}\label{Q:1}
	Suppose a strong embedding is associated the bigrading \( (i,j) \). What graph-theoretic information does this bigrading contain?
\end{question}
\begin{question}
	What other graph-theoretic objects appear in \( \dkh' ( \mathbb{K} ( \mathcal{G} ) ) \)?
\end{question}
Further lines of enquiry are opened by considering the place of doubled Lee homology in the hierarchy of homology theories. As mentioned above, doubled Lee homology is a perturbation of another theory, doubled Khovanov homology. Specifically, one obtains doubled Lee homology as the \( E_{\infty} \)-page of a spectral sequence whose \( E_2 \)-page is doubled Khovanov homology. It follows that doubled Lee homology is contained in doubled Khovanov homology.
\begin{question}
	What graph-theoretic information is contained in the larger theory, doubled Khovanov homology?
\end{question}

\subsection{Example}\label{Sec:example}
Consider the graphene \( \mathcal{G} \) and the bicolored multicycle \( C \)
\begin{equation*}
\includegraphics[scale=0.5]{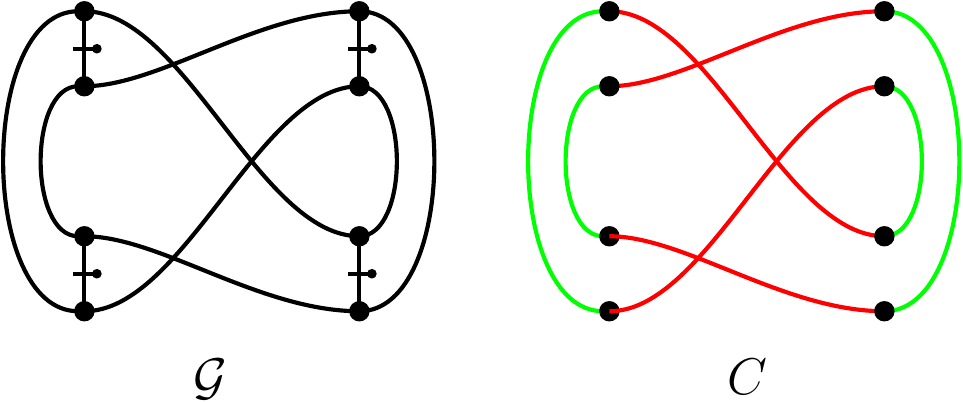}
\end{equation*}
The process described in \Cref{Sec:strong+cycles} associates to \( C \) the following strong embedding of \( \mathcal{G} \)
\begin{equation*}
\includegraphics[scale=0.4]{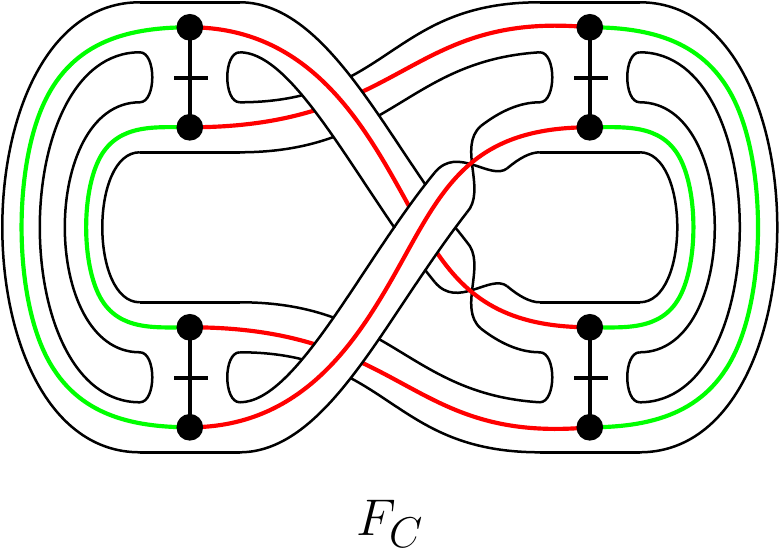}
\end{equation*}
(with discs attached to the boundary components).

Further, \( \mathbb{K}( \mathcal{G} ) = L \), where \( L \) is the link depicted in \Cref{Fig:2c}. The bicolored multicycle \( C \) is sent to the \(2\)-coloring depicted in \Cref{Fig:2c} under the bijection of \Cref{Prop:biject}. The generator of \(  \dkh' ( L ) \) associated to this \(2\)-coloring is of bigrading \( (-2,-5) \). It follows that the strong embedding depicted above is associated the bigrading \( (-2,-5) \). This is an instance of the setup of \Cref{Q:1}.

\bibliographystyle{plain}
\bibliography{library}

\end{document}